\documentclass[14pt, draft]{article}
\usepackage[cp1251]{inputenc}
\usepackage[ukrainian, russian, english]{babel}

\usepackage{ amsfonts, amssymb}
\usepackage{amsmath}
\usepackage{amsthm}

\newtheorem{theorem}{Теорема}

\newtheorem{lemma}{Лема}
\newtheorem{corollary}{Наслідок}

\theoremstyle{definition}

\begin{document}

\selectlanguage{ukrainian} \thispagestyle{empty}
 \pagestyle{myheadings}              %%%%%%%%%%%%%%%%%%%%% <--------------
%%%%%%%%%%%%%%%%%%%%%%%%%%%%%%%%%%%%%%%%%%%%%%%%%%%%%%%%%%%%%%%%%%%%%%%%%%%%%
%%%%%%%%%%%%%%%%%%%%%%%%%%%%%%%%%%%%%%%%%%%%%%%%%%%%%%%%%%%%%%%%%%%%%%%%%%%%%
%%%%%%%%%%%%%%%%%%%%%%%%%%%%%%%%%%%%%%%%%%%%%%%%%%%%%%%%%%%%%%%%%%%%%%%%%%%%%

\pagestyle{myheadings}              %%%%%%%%%%%%%%%%%%%%% <--------------
%%%%%%%%%%%%%%%%%%%%%%%%%%%%%%%%%%%%%%%%%%%%%%%%%%%%%%%%%%%%%%%%%%%%%%%%%%%%%
%%%%%%%%%%%%%%%%%%%%%%%%%%%%%%%%%%%%%%%%%%%%%%%%%%%%%%%%%%%%%%%%%%%%%%%%%%%%%
%%%%%%%%%%%%%%%%%%%%%%%%%%%%%%%%%%%%%%%%%%%%%%%%%%%%%%%%%%%%%%%%%%%%%%%%%%%%%

%{\small УДК 517.5} \vskip 3mm
УДК 517.5 \vskip 3mm

\noindent \bf А.С. Сердюк  \rm (Інститут математики НАН України, Київ) \\
\noindent \bf Т.А. Степанюк  \rm (Інститут математики НАН України, Київ)

\noindent {\bf A.S. Serdyuk} (Institute of Mathematics NAS of
Ukraine, Kyiv) \\
 \noindent {\bf T.A. Stepaniuk} (Institute of Mathematics NAS of
Ukraine, Kyiv)\\

 \vskip 5mm

{\bf Оцінки наближень інтерполяційними тригонометричними поліномами на класах згорток перiодичних функцiй високої гладкостi }

\vskip 5mm

{\bf Estimates of approximations  by interpolation trigonometric polynomials  on the classes of convolutions of high smoothness
}

\vskip 5mm

 \rm Встановлено  інтерполяційні аналоги нерівностей типу Лебега на множинах $C^{\psi}_{\beta}L_{1}$ $2\pi$--періодичних функцій $f$, котрі задаються згортками твірного ядра $\Psi_{\beta}(t)=\sum\limits_{k=1}^{\infty}\psi(k)\cos
\big(kt-\frac{\beta\pi}{2}\big)$, $\psi(k)\geq 0$, $\sum\limits_{k=1}^{\infty}\psi(k)<\infty$, $\beta\in\mathbb{R}$,
з функціями $\varphi$ із $L_{1}$.  В зазначених нерівностях при кожному $x\in\mathbb{R}$  модулі відхилень  $|f(x)- \tilde{S}_{n-1}(f;x)|$ інтерполяційних поліномів Лагранжа $ \tilde{S}_{n-1}(f;\cdot)$ оцінюються через найкращі наближення $E_{n}(\varphi)_{L_{1}}$ функцій $\varphi$ тригонометричними поліномами в $L_{1}$--метриках.
Коли послідовності $\psi(k)$ спадають до нуля швидше за довільну степеневу функцію, отримані нерівності в ряді важливих випадків є асимптотично точними. В таких випадках встановлено також асимптотичні рівності для точних верхніх меж поточкових наближень інтерполяційними тригонометричними поліномами  на класах згорток твірного ядра $\Psi_{\beta}$ із функціями $\varphi$, що належать одиничній кулі з простору $L_{1}$.
 
%В метриках пространств $L_{s}, \ 1\leq s\leq\infty$, найдены асимптотические равенства для  верхних граней приближений суммами Фурье  на классах
 % обобщенных интегралов Пуассона периодических функций, принадлежащих единичному %шару  пространства $L_{1}$.

\vskip 5mm

 \rm We establish interpolation analogues of Lebesgue type inequalities on the sets of $C^{\psi}_{\beta}L_{1}$ $2\pi$--periodic functions $f$, which are representable as convolutions of generating kernel  $\Psi_{\beta}(t)=\sum\limits_{k=1}^{\infty}\psi(k)\cos
\big(kt-\frac{\beta\pi}{2}\big)$, $\psi(k)\geq 0$, $\sum\limits_{k=1}^{\infty}\psi(k)<\infty$, $\beta\in\mathbb{R}$,
with functions $\varphi$ from $L_{1}$ .  In obtained inequalities for each  $x\in\mathbb{R}$ the modules of deviations   $|f(x)- \tilde{S}_{n-1}(f;x)|$ of interpolation Lagrange polynomials  $ \tilde{S}_{n-1}(f;\cdot)$  are estimated via best approximations $E_{n}(\varphi)_{L_{1}}$ of functions $\varphi$ by trigonometric polynomials in $L_{1}$--metrics. 
When the sequences  $\psi(k)$ decrease to zero faster than any power function, the obtained inequalities in many important cases are asymptotically exact. In such cases we also establish the asymptotic equalities for exact upper boundaries of pointwise approximations by interpolation trigonometric polynomials on the classes of convolutions of generating kernel $\Psi_{\beta}$ with functions $\varphi$, which belong to the unit ball from the space $L_{1}$.
 
\newpage

%%%%%%%%%%%%%%%%%%%%%%%%%%%%%%
%%%%%%%%%%%%%%%%%%%%%%%%%%%%%%%%%%

\section{Вступ}

Нехай $L_{p}$,
$1\leq p<\infty$, простір $2\pi$--періодичних сумовних у  $p$-му степені на
 $[0,2\pi)$ функцій $f$, з нормою
\begin{equation*}
\|f\|_{L_p}=\|f\|_{p}:=\Big(\int\limits_{0}^{2\pi}|f(t)|^{p}dt\Big)^{\frac{1}{p}};
\end{equation*}
 $L_{\infty}$ --- простір вимірних і суттєво обмежених   $2\pi$--періодичних функцій $f$ з нормою
 \begin{equation*}
\|f\|_{\infty}=\mathop{\rm{ess}\sup}\limits_{t}|f(t)|;
\end{equation*}
$C$ --- простір неперервних $2\pi$--періодичних функцій  $f$ з нормою
\begin{equation*}
{\|f\|_{C}=\max\limits_{t}|f(t)|}.
\end{equation*}

%Assume that $f$ is a $2\pi$--periodic function, summable on $[0,2\pi)$ with the Fourier series
%\begin{equation}\label{FourierSer}
%S[f]=\frac{a_{0}}{2}+\sum\limits_{k=1}^{\infty}\left( a_{k}\cos kx +b_{k}\sin kx \right)
%=\sum\limits_{k=0}^{\infty}A_{k}(f,x)
 %\end{equation}
 
Нехай $\psi(k)$ --- довільна фіксована послідовність  невід'ємних дійсних чисел, і  нехай  $\beta$ --- фіксоване дійсне число. 
Позначимо через $C^{\psi}_{\beta}L_{p}$ множину $2\pi$--періодичних функцій, які при всіх $x\in\mathbb{R}$ зображуються у вигляді згортки
\begin{equation}\label{conv}
f(x)=\frac{a_{0}}{2}+\frac{1}{\pi}\int\limits_{-\pi}^{\pi} \Psi_{\beta}(x-t)\varphi(t)dt,
\ a_{0}\in\mathbb{R}, \ \varphi\in L_{p}, \  \varphi\perp1 \
\end{equation}
з твірним ядром $\Psi_{\beta}$ вигляду
\begin{equation}\label{kernelPsi}
\Psi_{\beta}(t)=\sum\limits_{k=1}^{\infty}\psi(k)\cos
\big(kt-\frac{\beta\pi}{2}\big), \ \psi(k)\geq 0, \  \beta\in
    \mathbb{R},
\end{equation}
таким, що
 \begin{equation}\label{condition}
\sum\limits_{k=1}^{\infty}\psi(k)<\infty.
\end{equation}

Якщо функції $f$ і $\varphi$ пов'язані рівністю (\ref{conv}), то функцію $f$ в цьому співвідношенні називають $(\psi,\beta)$--похідною функції $f$ і позначають через $f^{\psi}_{\beta}$. З іншого боку функцію $f$ у рівності  (\ref{conv}) називають $(\psi, \beta)$--інтегралом функції $\varphi$ і позначають через $\mathcal{J}^{\psi}_{\beta}\varphi$. Поняття $(\psi, \beta)$--похідної  ($(\psi, \beta)$--інтеграла), які означалися в дещо інших термінах (за допомогою мультиплікаторів та зсувів аргументу), були введені О.І. Степанцем (див., наприклад, \cite{Stepanets1986_1}, \cite{Step monog 1987}, \cite{Stepanets1}).
Незважаючи на відмінність підходів до означення $(\psi,\beta)$--інтегралів та $(\psi,\beta)$--похідних у даній статті, порівняно, наприклад, із \cite{Stepanets1}, зазначимо, що за виконання \eqref{kernelPsi} i \eqref{condition} $(\psi,\beta)$--інтеграли $\mathcal{J}^{\psi}_{\beta}\varphi$ для будь-якої $\varphi\in L_{1}$ можуть відрізнятися між собою в залежності від вибору означення лише на множині міри нуль; аналогічно за виконання \eqref{kernelPsi} i \eqref{condition} і додаткової умови $\psi(k)>0$, $(\psi,\beta)$-- похідні $f^{\psi}_{\beta}$ для будь--якої функції $f$ вигляду \eqref{conv} можуть відрізнятися лише нам множині нульової міри.

Підмножину функцій $f$ з $C^{\psi}_{\beta}L_{p}$ таких, що $f^{\psi}_{\beta}\in B_{p}$,
де $B_{p}$ --- одинична куля в просторі $L_{p}$, тобто
\begin{equation*}
B_{p}:=\left\{\varphi: \ ||\varphi||_{p}\leq 1\right\},
\end{equation*}
будемо позначати через c $C^{\psi}_{\beta,p}$. Зрозуміло, що умова (\ref{condition}) гарантує неперервність твірного ядра $\Psi_{\beta}(t)$ вигляду (\ref{kernelPsi}), а отже і істинність вкладення 
 $C^{\psi}_{\beta}L_{p}\subset C \ \ (C^{\psi}_{\beta,p}\subset C)$.
 
 Якщо $\psi(k)=k^{-r}$, $r>0$, ядра $\Psi_{\beta}(t)$ є відомими ядрами Вейля--Надя $B_{r,\beta}(t)$ вигляду 
\begin{equation}\label{kernelWeylNady}
B_{r, \beta}(t)=\sum\limits_{k=1}^{\infty} k^{-r}\cos
\big(kt-\frac{\beta\pi}{2}\big),  \ r>0, \ \beta\in\mathbb{R}.
\end{equation} 

При цьому $(\psi,\beta)$--похідні $f^{\psi}_{\beta}$ майже скрізь збігаються з $(r,\beta)$--похідними $f^{r}_{\beta}$ в сенсі Вейля--Надя, а відповідні класи $C^{\psi}_{\beta,p}$ позначають через  $W^{r}_{\beta,p}$ (див. \cite[ $\S$ 6-8, Гл. 3]{Stepanets1}). Коли $r\in\mathbb{N}$, $r=\beta$, класи $W^{r}_{\beta,p}$ є відомими класами $W^{r}_{p}$   $2\pi$--періодичних функцій $f$, що мають абсолютно неперервні похідні $f^{(k)}$ до $(r-1)$--го порядку включно і для яких $f^{(r)}\in B_{p}$. В зазначеному випадку 
 $f^{r}_{\beta}$ майже скрізь збігаються з $f^{(r)}$.

У випадку, коли $\psi(k)=e^{-\alpha k^{r}}$, $\alpha>0$, $r>0$, ядра $\Psi_{\beta}(t)$ вигляду (\ref{kernelPsi}) є узагальненими ядрами Пуассона, тобто $\Psi_{\beta}(t)=P_{\alpha,r,\beta}(t)$, де
\begin{equation}\label{kernelPsi_GeneralizedPoisson}
P_{\alpha, r, \beta}(t)=\sum\limits_{k=1}^{\infty} e^{-\alpha k^{r}}\cos
\big(kt-\frac{\beta\pi}{2}\big), \ \alpha> 0, \ r>0, \ \beta\in\mathbb{R}.
\end{equation}
Породжені такими ядрами множини $C^{\psi}_{\beta}L_{p}$ та $C^{\psi}_{\beta,p}$ позначатимемо відповідно через  $C^{\alpha,r}_{\beta}L_{p}$ та  $C^{\alpha,r}_{\beta,p}$ і
називатемо множинами узагальнених інтегралів Пуассона, а $(\psi,\beta)$--похідні $f^{\psi}_{\beta}$  та $(\psi,\beta)$--інтеграли  $\mathcal{J}^{\psi}_{\beta}\varphi$ позначатимемо через $f^{\alpha,r}_{\beta}$  та  $\mathcal{J}^{\alpha,r}_{\beta}\varphi$ відповідно.

Для довільних $r>0$ множини $C^{\alpha,r}_{\beta}L_{p}$ та $C^{\alpha,r}_{\beta,p}$  є підмножинами множини $2\pi$--періодичних нескінченно диференційовних функцій $D^{\infty}$, тобто $C^{\alpha,r}_{\beta,p}\subset C^{\alpha,r}_{\beta}L_{p}  \subset D^{\infty}$ (див., наприклад, \cite{Stepanets1},  \cite{Stepanets_Serdyuk_Shydlich2007}). При $r=1$ множини $C^{\alpha,r}_{\beta}L_{p} \,(C^{\alpha,r}_{\beta,p})$ є множинами  інтегралів Пуассона гармонічних функцій і складаються з $2\pi$--періодичних аналітичних функцій, що допускають регулярне продовження  у смугу $|\mathrm{Im} \, z|< \alpha$ комплексної площини (див., наприклад, \cite[c. 141]{Stepanets1}). При $r>1$  множини $C^{\alpha,r}_{\beta}L_{p} (C^{\alpha,r}_{\beta,p})$ є множинами  $2\pi$--періодичних цілих функцій  (див., наприклад,   \cite[c. 142]{Stepanets1}).
Зв'язок між множинами $C^{\alpha,r}_{\beta}L_{p}$ і відомими класами Жевре  
 \begin{equation*}
J_{a}=\left\{f\in D^{\infty}: \sup\limits_{k\in\mathbb{N}}\left( \frac{\|f^{(k)} \|_{C}}{(k!)^{a}}\right)^{1/k}<\infty \right\}, \ \ a>0,
\end{equation*}
вивчався у роботі \cite{Stepanets_Serdyuk_Shydlich2009}.

Нами досліджуються апроксимативні властивості множин $C^{\psi}_{\beta}L_{1} $ та $C^{\psi}_{\beta,1} $ коли у ролі агрегатів наближення виступають класичні інтерполяційні тригонометричні поліноми Лагранжа, що задані непарним числом рівномірно розподілених вузлів інтерполяції.

Для будь--якої функції $f(x)$ із $ C$ через $\widetilde{S}_{n-1}(f;x)$, $n\in\mathbb{N}$, будемо позначати тригонометричний поліном порядку $n-1$, що інтерполює $f(x)$ у вузлах $x_{k}^{(n-1)}=\frac{2k\pi}{2n-1}$, $k\in\mathbb{Z}$, тобто такий, що
\begin{equation}\label{InterpolationPolynomS}
\tilde{S}_{n-1}(f;x_{k}^{(n-1)})=f(x_{k}^{(n-1)}), \ k = 0,1, ..., 2n-2.
\end{equation}

Поліном  $\widetilde{S}_{n-1}(f;\cdot)$ однозначно задається інтерполяційними умовами \eqref{InterpolationPolynomS}, називається інтерполяційним поліномом Лагранжа і може бути зображений в явному вигляді через ядро Діріхлє
\begin{equation*}
D_{n-1}(t)=\frac{1}{2}+\sum\limits_{k=1}^{n-1}\cos kt = \frac{\sin (n-\frac{1}{2})t }{2\sin \frac{t}{2}}
\end{equation*}
наступним чином:
\begin{equation}\label{InterpolationPolynomS_Dirichlet}
\tilde{S}_{n-1}(f;x)= \frac{2}{2n-1}\sum\limits_{k=0}^{2n-2}f(x_{k}^{(n-1)}) D_{n-1}(x-x_{k}^{(n-1)}).
\end{equation}

Нехай $\mathcal{T}_{2n-1}$ --- простір усіх тригонометричних поліномів $t_{n-1}$ порядку $n-1$ і $E_{n}(f)_{L_{p}}$ --- найкраще наближення функції $f\in L_{p}$, $1\leq p\leq \infty$, в $L_{p}$--метриці тригонометричними поліномами $t_{n-1}\in \mathcal{T}_{2n-1}$, тобто величина
\begin{equation*}
E_{n}(f)_{L_{p}}=\inf\limits_{t_{n-1}\in \mathcal{T}_{2n-1}}\|f-t_{n-1}\|_{p}, 
\end{equation*}
а $E_{n}(f)_{C}$ --- найкраще рівномірне наближення функціі $f\in C$ тригонометричними поліномами $t_{n-1}$, тобто величина
\begin{equation*}
E_{n}(f)_{C}=\inf\limits_{t_{n-1}\in \mathcal{T}_{2n-1}}\|f-t_{n-1}\|_{C}.
\end{equation*}

Позначимо через $\tilde{\rho}_{n}(f;\cdot)$ відхилення від функції $f\in C$ її інтерполяційного полінома Лагранжа $\tilde{S}_{n-1}(f;\cdot)$ 
\begin{equation}\label{rhoDef}
 \tilde{\rho}_{n}(f;x)=f(x)- \tilde{S}_{n-1}(f;x).
\end{equation}

Для модулів величин вигляду \eqref{rhoDef} має місце нерівність (див., наприклад, \cite{Bernshteyn},  \cite{Korn}, \cite{Stepanets2})
\begin{equation}\label{LebesgueIneq}
\left| f(x)- \tilde{S}_{n-1}(f;x)\right| \leq
(1+ \bar{L}_{n}(x))E_{n}(f)_{C}, \ \ f\in C, \ \ x\in\mathbb{R},
\end{equation}
де
\begin{equation}\label{bar_Ln}
 \bar{L}_{n}(x)
 = \frac{2}{2n-1} \sum\limits_{k=0}^{2n-2}\left|D_{n-1}(x-x_{k}^{(n-1)}) \right|.
\end{equation}

Нерівність \eqref{LebesgueIneq} є інтерполяційним аналогом класичної нерівності Лебега для сум Фур'є. В ній функцію $\bar{L}_{n}(x)$ вигляду \eqref{bar_Ln} називають функцією Лебега інтерполяційного оператора $\tilde{S}_{n-1}$ вигляду \eqref{InterpolationPolynomS_Dirichlet}.

Асимптотичну поведінку функції Лебега $\bar{L}_{n}(x)$ при $n\rightarrow \infty$ описує наступна формула:
\begin{equation}\label{bar_Ln_Asymp}
 \bar{L}_{n}(x)
 = \frac{2}{\pi} \left| \sin \frac{2n-1}{2}x\right| \ln n + \mathcal{O}(1), \ \ x\in\mathbb{R},
\end{equation}
в якій $ \mathcal{O}(1)$ --- величина, що рівномірно обмежена по $x$ і по $n$ (див., наприклад, \cite[С. 66]{Korn}, \cite{Nikolsky1940}). 

З урахуванням \eqref{bar_Ln_Asymp} нерівність \eqref{LebesgueIneq} можна записати у вигляді

\begin{equation}\label{LebesgueIneq2}
\left| \tilde{\rho}_{n}(f;x) \right| \leq
\left(\frac{2}{\pi} \left| \sin \frac{2n-1}{2}x\right| \ln n + \mathcal{O}(1) \right)
E_{n}(f)_{C},  \ \ f\in C, \ \ x\in\mathbb{R}.
\end{equation}

Незважаючи на простоту й загальність, нерівність \eqref{LebesgueIneq2}, як довів С.М. Нікольський \cite{Nikolsky1945}, є асимптотично точною на класах диференційовних функцій $W^{r}_{\infty}$, $r\in\mathbb{N}$.

Однак для класів нескінченно диференційовних, аналітичних чи цілих функцій, оцінки відхилень $\left| \tilde{\rho}_{n}(f;x) \right|$, що базуються на використанні нерівностей  \eqref{LebesgueIneq} (чи \eqref{LebesgueIneq2}), перестають бути асимптотично точними і навіть можуть бути не точними за порядком.

У даній роботі для функцій з множин   $C^{\psi}_{\beta}L_{1}$,  $\beta \in \mathbb{R}$,  встановлено інтерполяційні аналоги нерівностей типу Лебега, в яких оцінки зверху величин $|\tilde{\rho}_{n}(f;x)|$, $x\in\mathbb{R}$,
виражаються через найкращі наближення $E_{n}(f^{\psi}_{\beta})_{L_{1}}$.
В ній також доведено асимптотичну непокращуваність отриманих нерівностей на множинах $C^{\psi}_{\beta}L_{1}$ у випадках, коли послідовності $\psi(k)$ підпорядковані умові
\begin{equation}\label{LimRelation}
\lim\limits_{n\rightarrow\infty}\frac{ \frac{1}{n}\sum\limits_{k=1}^{\infty} k\psi(k+n)}{  \sum\limits_{k=n}^{\infty}\psi(k)} = 0.
\end{equation}
Як показано в \cite{SerdyukStepanyuk_UMJ_1_2023}, умову \eqref{LimRelation} задовольняє низка важливих послідовностей $\psi(k)$, які прямують до нуля при $k\rightarrow \infty$ швидше за довільну степеневу функцію.

Також у роботі  знайдено розв'язок задачі Колмогорова--Нікольського 
   для інтерполяційних поліномів Лагранжа $ \tilde{S}_{n-1}(f;x) $ вигляду \eqref{InterpolationPolynomS_Dirichlet}  на класах   $C^{\psi}_{\beta,1}$, за умови, що $\psi$ задовольняють умову \eqref{LimRelation}, тобто встановлено асимптотичні при
 $n\rightarrow\infty$ рівності для величин
\begin{equation}\label{quantityInterpol}
\tilde{\mathcal{E}}_{n}(C^{\psi}_{\beta,1};x)=\sup\limits_{f\in C^{\psi}_{\beta,1} } \left| \tilde{\rho}_{n}(f;x) \right|, \ \ x\in\mathbb{R}.
\end{equation}

Зазначимо, що точні порядкові рівності для величин
\begin{equation*}
\tilde{\mathcal{E}}_{n}(C^{\psi}_{\beta,p})_{L_{s}}=\sup\limits_{f\in C^{\psi}_{\beta,p} } \left \| f-\tilde{S}_{n-1}(f) \right\|_{s}
\end{equation*}
та
\begin{equation*}
\tilde{\mathcal{E}}_{n}(C^{\psi}_{\beta,p})_{C}=\sup\limits_{f\in C^{\psi}_{\beta,p} } \left\| f-\tilde{S}_{n-1}(f) \right\|_{C}
\end{equation*}
при різних співвідношеннях між параметрами $p$ i $s$, $0<p,s<\infty$, випливають із робіт \cite{L_Sharapudinov}
i
\cite{L_Oskolkov_1986}.

Вивченню  апроксимативних властивостей сум
$\widetilde{S}_{n-1}(f)$, на множинах
$(\psi,\beta)$--диференційовних функцій, присвячені роботи 
  \cite{Serdyuk_1998}--\cite{L_Serdyuk_2004_nd},
  \cite{Serdyuk2012}--\cite{SerdyukSokolenko2019},
\cite{StepanetsSerdyuk2000Zb},
\cite{StepanetsSerdyuk2000}.

 Зокрема, у роботах \cite{SerdyukDopov1999}, \cite{L_Serdyuk_2004_nd}, \cite{StepanetsSerdyuk2000Zb},  \cite{StepanetsSerdyuk2000} було встановлено асимптотичні рівності для величин $\tilde{\mathcal{E}}_{n}(C^{\psi}_{\beta,\infty};x)$, коли $C^{\psi}_{\beta,\infty}$ є класами нескінченно диференційовних, аналітичних та цілих функцій.
 У роботах \cite{Serdyuk_1998}, \cite{L_Serdyuk 2000 c}, \cite{L_Serdyuk 2001 nd}, \cite{L_Serdyuk_2002_a} знайдено аналогічні рівності для величин  $\tilde{\mathcal{E}}_{n}(C^{\psi}_{\beta,1})_{L_{1}}$. Для класів аналітичних та цілих функцій 
$C^{\psi}_{\beta,p}$, $1\leq p\leq \infty$, асимптотичні рівності для величин $\tilde{\mathcal{E}}_{n}(C^{\psi}_{\beta,p}; x)$ були встановлені в \cite{Serdyuk2012}, \cite{SerdyukVoitovych2010}. На класах цілих функцій $C^{\psi}_{\beta,1}$  точна асимптотика величин $\tilde{\mathcal{E}}_{n}(C^{\psi}_{\beta,1})_{L_{p}}$ при $1\leq p\leq \infty$ була встановлена в \cite{SerdyukSokolenko2019}.
Крім того, у роботах \cite {SerdyukSokolenko2016}, \cite{SerdyukSokolenko2017} знайдено точні значення величин 
$\tilde{\mathcal{E}}_{n}(C^{\psi}_{\beta,2};x)$.
Незважаючи на інтенсивність досліджень в даному напрямі, в загальному питання про точну асимптотику величин $\tilde{\mathcal{E}}_{n}(C^{\psi}_{\beta,1};x)$ при $n\rightarrow\infty$ у випадку, коли послідовності $\psi(k)$ спадають до нуля швидше за будь-яку степеневу функцію, але повільніше за будь-яку геометричну прогресію, за включенням деяких конкретних випадків, до цих пір залишалось відкритим.

Зазначимо також, що дана робота тісно пов'язана із результатами роботи авторів \cite{SerdyukStepanyuk_UMJ_1_2023} в якій за виконання умови \eqref{LimRelation} було знайдено асимптотично непокращувані нерівності типу Лебега  для сум Фур'є на множинах $C^{\psi}_{\beta}L_{1}$, а також знайдено розв'язок задачі Колмогорова--Нікольського для сум Фур'є $S_{n-1}$ на класах $C^{\psi}_{\beta,1}$, тобто досліджено точну асимптотику величин
\begin{equation*}
{\mathcal{E}}_{n}(C^{\psi}_{\beta,1})_{C}=\sup\limits_{f\in C^{\psi}_{\beta,p} } \left\| f-{S}_{n-1}(t) \right\|_{C}
\end{equation*}
при $n\rightarrow \infty$. Проблемам, пов'язаним із розв'язанням цієї задачі присвячені роботи
\cite{Nikolsky 1946},
\cite{Serdyuk2005},
\cite{Serdyuk2005Lp},
\cite{SerdyukSokolenkoMFAT2019}--\cite{Stepanets1},
\cite{StepanetsSerdyuk2000No3},
\cite{StepanetsSerdyuk2000No6},
\cite{Stechkin 1980},
\cite{Teljakovsky1989}.

%\cite{SerdyukDopov1999}--\cite{SerdyukSokolenko2022}?????.

\section{Наближення  інтерполяційними поліномами Лагранжа на множинах згорток періодичних функцій}

\begin{theorem}\label{Interpolation_2023_theorem_GeneralCase}
Нехай $\sum\limits_{k=1}^{\infty}k\psi(k)<\infty$, $\psi(k)\geq 0$, $k=1,2,...$, $\beta\in\mathbb{R}$ i $n\in\mathbb{N}$.  Тоді, для всіх $x\in\mathbb{R}$ і довільної  функції
 $f\in C^{\psi}_{\beta}L_{1}$  має місце нерівність
\begin{align}\label{interpolation_generalCase_C^psi_Inequality}
|\tilde{\rho}_{n}(f;x)|
\leq
\frac{2}{\pi} \left| \sin \frac{2n-1}{2}x\right| 
\left( \sum\limits_{k=0}^{\infty}\sum\limits_{\nu=(2k+1)n-k}^{\infty} \psi(\nu) \right)
E_{n}(f^{\psi}_{\beta})_{L_{1}}.
\end{align}%where $F(a,b;c;d)$ is Gauss hypergeometric function.
Крім того, для довільної функції $f\in C^{\psi}_{\beta}L_{1}$ можна вказати функцію $\mathcal{F}(\cdot)=\mathcal{F}(f;n;x, \cdot)$ таку, що $E_{n}(\mathcal{F}^{\psi}_{\beta})_{L_1}=E_{n}(f^{\psi}_{\beta})_{L_1}$  і для якої виконується наступна рівність:
\begin{align}\label{interpolation_generalCase_C^psi_Equality}
|\tilde{\rho}_{n}(\mathcal{F}; x)|
=
 \frac{2}{\pi} \left| \sin\frac{2n-1}{2}x \right|
 \left( 
\sum\limits_{k=n}^{\infty}\psi(k)+
\frac{ \xi }{n}\sum\limits_{k=1}^{\infty} k\psi(k+n) \right)
E_{n}(f^{\psi}_{\beta})_{L_{1}}.
\end{align}
В \eqref{interpolation_generalCase_C^psi_Equality} величина $\xi=\xi(f;n;\psi;\beta;x)$  є такою, що $-\left(1+2\pi \right)\leq \xi \leq 1$.
\end{theorem}

\begin{proof}

Згідно з лемою 1 роботи \cite{StepanetsSerdyuk2000} для довільної функції $f\in C^{\psi}_{\beta}L_{1}$, $\psi(k)\geq0$, $\sum\limits_{k=1}^{\infty}k\psi(k)<\infty$, $\beta\in\mathbb{R}$ у кожній точці $x\in\mathbb{R}$ має місце наступне інтегральне зображення величини $\tilde{\rho}_{n}(f;x)$:
\begin{equation}\label{IntegrRepr}
\tilde{\rho}_{n}(f;x)=\frac{2}{\pi}\sin\frac{2n-1}{2}x 
\int\limits_{-\pi}^{\pi}\delta_{n}(t+x)
\left(\sum\limits_{k=n}^{\infty}\psi(k)\cos (kt+\gamma_{n})+r_{n}(t)
 \right) dt,
\end{equation}
в якому $\delta_{n}(\tau)= f^{\psi}_{\beta}(\tau)-t_{n-1}(\tau)$, $t_{n-1}$ --- довільний тригонометричний поліном із множини $\mathcal{T}_{2n-1}$, а $r_{n}$ i $\gamma_{n}$ означені за допомогою рівностей 
\begin{equation}\label{rn}
r_{n}(t)=r_{n}(\psi;\beta;x;t)=
\sum\limits_{k=1}^{\infty}\sum\limits_{\nu=(2k+1)n-k}^{\infty}\!\!\! \psi(\nu)\sin\left(\nu t+ \left(k+\frac{1}{2} \right)(2n-1)x+ \frac{\beta\pi}{2}\right)
\end{equation}
і
\begin{equation}\label{gamma_n}
\gamma_{n}=\gamma_{n}(\beta;x)=
\frac{(2n-1)x+\pi(\beta-1)}{2}
\end{equation}
відповідно.

Зауважимо, що в лемі 1 зазначеної роботи \cite{StepanetsSerdyuk2000} замість умови $\psi(k)\geq0$ фігурувала більш жорстка умова $\psi(k)>0$. Втім, проаналізувавши  доведення згаданої леми, легко переконатись, що вона залишається вірною і за умови $\psi(k)\geq0$.

Беручи в \eqref{IntegrRepr} в якості $t_{n-1}$ поліном $t_{n-1}^{*}$ найкращого наближення у просторі $L_{1}$ функції $f^{\psi}_{\beta}$, тобто такий, що
\begin{equation}\label{bestApprox}
\| f^{\psi}_{\beta} - t_{n-1}^{*} \|_{1} = E_{n}(f^{\psi}_{\beta} )_{L_{1}}
=\inf\limits_{t_{n-1} \in \mathcal{T}_{2n-1}}\| f^{\psi}_{\beta} - t_{n-1} \|_{1} ,
\end{equation}
і застосовуючи нерівність Гельдера (див., наприклад, \cite[с. 391]{Korn})
\begin{equation}\label{HolderIneq}
\int\limits_{-\pi}^{\pi}|h(t)g(t)|dt \leq \|h\|_{p}\|g\|_{p'}, \ \ h\in L_{p}, \ \ 1\leq p\leq \infty, \ \ 
g\in L_{p'}, \ \ \frac{1}{p}+\frac{1}{p'}=1,
\end{equation}
при $p=\infty$, для довільної функції $f\in C^{\psi}_{\beta}L_{1}$ в силу \eqref{IntegrRepr}--\eqref{HolderIneq} отримуємо

\begin{align}\label{AdditionEstimRho}
|\tilde{\rho}_{n}(f;x)| \leq &  \frac{2}{\pi} \left| \sin \frac{2n-1}{2}x\right| 
\left( \left\| \sum\limits_{k=n}^{\infty}\psi(k)\cos (kt+\gamma_{n})   \right\|_{\infty} +\| r_{n}(t)\|_{\infty}
\right)E_{n}(f^{\psi}_{\beta})_{L_{1}} \notag \\
 \leq &
\frac{2}{\pi} \left| \sin \frac{2n-1}{2}x\right| 
\left( \sum\limits_{k=n}^{\infty}\psi(k)+\sum\limits_{k=1}^{\infty}\sum\limits_{\nu=(2k+1)n-k}^{\infty} \psi(\nu)
\right)E_{n}(f^{\psi}_{\beta})_{L_{1}}
 \notag \\
 = &
\frac{2}{\pi} \left| \sin \frac{2n-1}{2}x\right| 
\left( \sum\limits_{k=0}^{\infty}\sum\limits_{\nu=(2k+1)n-k}^{\infty} \psi(\nu)
\right)
E_{n}(f^{\psi}_{\beta})_{L_{1}}
\end{align}
де  $r_{n}$  та $\gamma_{n}$ означені формулами \eqref{rn} i \eqref{gamma_n}  відповідно.

Далі доведемо  справедливість другої частини теореми~\ref{Interpolation_2023_theorem_GeneralCase}.
Користуючись інтегральним зображенням \eqref{IntegrRepr} та беручи до уваги ортогональність функції $r_{n}(t)$ вигляду \eqref{rn} до будь-якого тригонометричного полінома $t_{n}\in\mathcal{T}_{2n-1}$, для довільної функції $f\in C^{\psi}_{\beta}L_{1}$,  в кожній точці $x\in \mathbb{R}$ можемо записати

\begin{align}\label{IntegrRepr1}
&\tilde{\rho}_{n}(f;x)=
f(x)-\tilde{S}_{n-1}(f;x) \notag \\
=&2\sin\frac{2n-1}{2}x 
\left( \frac{1}{\pi}\int\limits_{-\pi}^{\pi}f^{\psi}_{\beta}(t+x)
\sum\limits_{k=n}^{\infty}\psi(k)\cos (kt+\gamma_{n})
 dt
 +
 \frac{1}{\pi}\int\limits_{-\pi}^{\pi}\delta_{n}(t+x)
r_{n}(t)
 dt \right),
\end{align}
де  $\delta_{n}(\cdot)= f^{\psi}_{\beta}(\cdot)-t_{n-1}(\cdot)$, $t_{n-1}$ --- довільний  поліном з $\mathcal{T}_{2n-1}$, а $r_{n}(t)$ i $\gamma_{n}$ означені за допомогою рівностей \eqref{rn} і \eqref{gamma_n} відповідно.

Із формул \eqref{rn}--\eqref{HolderIneq} випливає оцінка
\begin{align}\label{IntegrRepr_Estimate}
&\left|\int\limits_{-\pi}^{\pi}\delta_{n}(t+x) r_{n}(t)dt \right|
=
\left|\int\limits_{-\pi}^{\pi}\left(f^{\psi}_{\beta}(t+x)-t_{n-1}^{*}(t+x) \right)r_{n}(t)dt \right|
\notag\\
& \leq \|r_{n}\|_{\infty} E_{n}(f^{\psi}_{\beta})_{L_{1}}
\leq \left( \sum\limits_{k=1}^{\infty} \sum\limits_{\nu=(2k+1)n-k}^{\infty} \psi(\nu)\right) E_{n}(f^{\psi}_{\beta})_{L_{1}}.
\end{align}

А, отже, в силу \eqref{IntegrRepr_Estimate} можна записати
\begin{align}\label{IntegrRepr_Estimate_1}
\int\limits_{-\pi}^{\pi}\delta_{n}(t+x)r_{n}(t)dt 
= \xi_{1} \left( \sum\limits_{k=1}^{\infty} \sum\limits_{\nu=(2k+1)n-k}^{\infty} \psi(\nu)\right) E_{n}(f^{\psi}_{\beta})_{L_{1}},
\end{align}
де $\xi_{1}= \xi_{1}(f;n;\psi,\beta;x)$ така, що $-1< \xi_{1} \leq 1$.

Для функції 
\begin{equation}\label{g_x}
g_{x}(\cdot):=\frac{1}{\pi}\int\limits_{-\pi}^{\pi}f^{\psi}_{\beta}(t+\cdot)
\sum\limits_{k=1}^{\infty}\psi(k) \cos (kt+\gamma_{n})
 dt,
\end{equation}
яка очевидно належить до множини $C^{\psi}_{2\gamma_{n}/\pi}L_{1}$, при будь--якому фіксованому $x\in\mathbb{R}$  відхилення $\rho(g_{x}; \cdot)$ її частинних сум Фур'є  $S_{n-1}(g_{x}, \cdot)$  підпорядковане рівності
\begin{equation}\label{rho_n}
\rho(g_{x}; \cdot)=
g_{x}(\cdot)-S_{n-1}(g_{x}, \cdot)
=
\frac{1}{\pi}\int\limits_{-\pi}^{\pi}f^{\psi}_{\beta}(t+\cdot)
\sum\limits_{k=n}^{\infty}\psi(k)\cos (kt+\gamma_{n})
 dt,
\end{equation}
і, зокрема,
\begin{equation}\label{rho_nx}
\rho(g_{x}; x)=
g_{x}(x)-S_{n-1}(g_{x}, x)
=
\frac{1}{\pi}\int\limits_{-\pi}^{\pi}f^{\psi}_{\beta}(t+x)
\sum\limits_{k=n}^{\infty}\psi(k) \cos (kt+\gamma_{n})
 dt.
\end{equation}

Відповідно до теореми 2.1 роботи \cite{SerdyukStepanyuk_UMJ_1_2023} для функції $g_{x}(\cdot)$ з множини $C^{\psi}_{2\gamma_{n}/\pi}L_{1}$ при кожному $n\in\mathbb{N}$ знайдеться функція $G(\cdot)= G(f,n;x;\cdot)$ така, що
\begin{equation}\label{G_derivative}
E_{n}(G^{\psi}_{2\gamma_{n}/\pi})_{L_{1}}=E_{n}(f^{\psi}_{\beta})_{L_{1}}
\end{equation}
і для якої 
\begin{align}\label{th2Eq1}
\| \rho_{n}(G,\cdot)\|_{C}=&
\| G(\cdot) - S_{n-1}(G; \cdot)\|_{C} \notag \\
=& \left( \frac{1}{\pi}
\sum\limits_{k=n}^{\infty}\psi(k)+
 \frac{\xi_{2}}{n} \sum\limits_{k=1}^{\infty}k\psi(k+n)\right)E_{n}(f^{\psi}_{\beta})_{L_{1}},
\end{align}
де величина $\xi_{2}=\xi_{2}(f;n;\psi;\beta;x)$ така, що $-2\leq \xi_{2} \leq 0$.

Виберемо точку $x_{0}$ таким чином, щоб виконувалась рівність
\begin{equation}\label{G_equality1}
| {\rho}_{n}(G;x_{0})|=\| {\rho}_{n}(G;\cdot)\|_{C}.
\end{equation}
 Розглянемо функцію $\mathcal{F}(t)$, означену рівністю
 \begin{equation}\label{F_equality1}
\mathcal{F}(t):= \mathcal{J}^{\psi}_{\beta} G^{\psi}_{2\gamma_{n}/\pi}(t-x+x_{0}),
\end{equation}
 яка очевидно належить множині  $C^{\psi}_{\beta}L_{1}$ і покажемо, що вона є шуканою функцією.
  Для функції  $\mathcal{F}(t)$, з урахуванням формул \eqref{G_derivative}, \eqref{F_equality1} та інваріантності $L_{1}$-норми відносно зсуву аргументу маємо
 \begin{equation}\label{G_derivative1}
E_{n}(\mathcal{F}^{\psi}_{\beta})_{L_{1}}=E_{n}(G^{\psi}_{2\gamma_{n}/\pi})_{L_{1}}=E_{n}(f^{\psi}_{\beta})_{L_{1}}.
\end{equation}

Отже, в силу \eqref{IntegrRepr1}, \eqref{IntegrRepr_Estimate_1}, \eqref{G_derivative}, \eqref{th2Eq1}, %\eqref{rho_G_Theorem_p=1},
 \eqref{G_equality1} i \eqref{G_derivative1}
 %\eqref{rNAdditEstimate}, \eqref{HolderIneq},  
  для довільного заданого значення аргументу $x\in\mathbb{R}$ 

\begin{align}\label{rho_F_Theorem_Interpolation}
& |\tilde{\rho}_{n}(\mathcal{F}; x)|
 \notag \\
 = & 2 \left| \sin\frac{2n-1}{2}x \right|
 \!\!
\left( \left| \frac{1}{\pi}\!\! \int\limits_{-\pi}^{\pi} \!\!  G^{\psi}_{2\gamma_{n}/\pi}(x_{0}\!+\! t)
\sum\limits_{k=n}^{\infty} \!\psi(k)\cos\left(k t\!+\!\gamma_{n}\right)
 dt \right| \!+\! \frac{\xi_{1} }{\pi} \!\! \left( \sum\limits_{k=1}^{\infty} \!\sum\limits_{\nu=(2k+1)n-k}^{\infty} \!\!\! \!\!\psi(\nu)
\!\! \right) \!\!
E_{n}(f^{\psi}_{\beta})_{L_{1}} \!\!\right)
  \notag \\
  = &2\left| \sin\frac{2n-1}{2}x \right|
\left( |\rho_{n}(G,x_{0})|
 +
\frac{\xi_{1} }{\pi} \left( \sum\limits_{k=1}^{\infty}\sum\limits_{\nu=(2k+1)n-k}^{\infty} \psi(\nu)
 \right)
E_{n}(f^{\psi}_{\beta})_{L_{1}}\right)
 \notag \\
 = &2\left| \sin\frac{2n-1}{2}x \right|
\left(\|\rho_{n}(G,\cdot)\|_{C}
 +
\frac{\xi_{1} }{\pi} \left( \sum\limits_{k=1}^{\infty}\sum\limits_{\nu=(2k+1)n-k}^{\infty} \psi(\nu)
 \right)
E_{n}(f^{\psi}_{\beta})_{L_{1}}\right)
 \notag \\
 = &2\left| \sin\frac{2n-1}{2}x \right|
 \!\!
\left( \! \!\left( \frac{1}{\pi}
\sum\limits_{k=n}^{\infty}\psi(k) \!+ \!
 \frac{\xi_{2}}{n} \sum\limits_{k=1}^{\infty}k\psi(k+n)\right) \!\! E_{n}(\varphi)_{L_{1}}
 \!\!+\!\!
\frac{\xi_{1} }{\pi} \left( \sum\limits_{k=1}^{\infty}\sum\limits_{\nu=(2k+1)n-k}^{\infty} \!\!\!\! \psi(\nu)
\!\!
 \right) \!\!
E_{n}(f^{\psi}_{\beta})_{L_{1}} \!\!\right)
\notag \\
 = &2\left| \sin\frac{2n-1}{2}x \right|
 \left( \frac{1}{\pi}
\sum\limits_{k=n}^{\infty}\psi(k)+
\frac{\xi_{2}}{n} \sum\limits_{k=1}^{\infty} k\psi(k+n) +\frac{\xi_{1} }{\pi} \left( \sum\limits_{k=1}^{\infty}\sum\limits_{\nu=(2k+1)n-k}^{\infty} \!\! \!\psi(\nu)
 \right) \!\! \right)\!\!
E_{n}(f^{\psi}_{\beta})_{L_{1}},
\end{align}
де $-2\leq \xi_{2}\leq 0$,   $-1\leq \xi_{1}\leq 1$.

В подальшому нам буде корисним наступне твердження.
\begin{lemma}\label{Lemma_SumInterpolationPart_Label}
Нехай $\sum\limits_{k=1}^{\infty}k\psi(k)<\infty$, $\psi(k)\geq 0$, $n\in\mathbb{N}$.
Тоді
\begin{equation}\label{Lemma_SumInterpolationPart_Inequality}
\frac{1}{n}\sum\limits_{k=1}^{\infty}k\psi(k+n)\geq \sum\limits_{k=1}^{\infty}\sum\limits_{\nu=(2k+1)n-k}^{\infty} \psi(\nu).
\end{equation}
\end{lemma}
\begin{proof}
Нерівність \eqref{Lemma_SumInterpolationPart_Inequality} випливає з наступного ланцюжка перетворень
\begin{align*}
&\frac{1}{n}\sum\limits_{k=1}^{\infty}k\psi(k+n)=
\frac{1}{n} \sum\limits_{k=0}^{\infty}\sum\limits_{\nu=(2k+1)n-k}^{(2k+1)n-k +2n-2}(\nu-n)\psi(\nu)
\notag \\
=& \frac{1}{n} \sum\limits_{k=0}^{\infty}\sum\limits_{\nu=0}^{2n-2}(2kn-k+\nu)\psi((2k+1)n-k+\nu)
\notag\\
&\geq \frac{1}{n} \sum\limits_{k=0}^{\infty}\sum\limits_{\nu=0}^{2n-2}(2n-1)k\psi((2k+1)n-k+\nu) \notag\\
=& \frac{2n-1}{n} \sum\limits_{k=1}^{\infty}k \sum\limits_{\nu=0}^{2n-2} \psi((2k+1)n-k+\nu)
= \frac{2n-1}{n} \sum\limits_{k=1}^{\infty}k \sum\limits_{\nu=(2k+1)n-k}^{(2k+1)n-k+2n-2} \psi(\nu) \notag \\
= & \frac{2n-1}{n} \sum\limits_{k=1}^{\infty}k \sum\limits_{\nu=(2k+1)n-k}^{\infty} \psi(\nu) 
=\left(2-\frac{1}{n} \right) \sum\limits_{k=1}^{\infty} \sum\limits_{\nu=(2k+1)n-k} \psi(\nu)
\geq
 \sum\limits_{k=1}^{\infty}\sum\limits_{\nu=(2k+1)n-k}^{\infty} \psi(\nu).
\end{align*}
Лему доведено.
\end{proof}
В силу \eqref{Lemma_SumInterpolationPart_Inequality}
\begin{equation}\label{Lemma_SumInterpolationPart_Inequality_v1}
\frac{\xi_{1}}{\pi} \sum\limits_{k=1}^{\infty}\sum\limits_{\nu=(2k+1)n-k}^{\infty} \psi(\nu)+\frac{\xi_{2}}{n}\sum\limits_{k=1}^{\infty}k\psi(k+n)
=\frac{\xi}{\pi n}\sum\limits_{k=1}^{\infty}k\psi(k+n),
\end{equation}
де $\xi=\xi(f;n;\psi;\beta;x)$ така, що $-\left(1+2\pi \right)\leq \xi\leq 1$.

З \eqref{rho_F_Theorem_Interpolation} i \eqref{Lemma_SumInterpolationPart_Inequality_v1} випливає \eqref{interpolation_generalCase_C^psi_Equality}. Теорему~\ref{Interpolation_2023_theorem_GeneralCase} доведено.\end{proof}

Зауважимо, що з урахуванням формули \eqref{Lemma_SumInterpolationPart_Inequality} леми~\ref{Lemma_SumInterpolationPart_Label}, нерівність \eqref{interpolation_generalCase_C^psi_Inequality}  теореми~\ref{Interpolation_2023_theorem_GeneralCase} приводить до наступної нерівності:
\begin{align}\label{interpolation_generalCase_C^psi_Inequality_modified}
|\tilde{\rho}_{n}(f;x)|
\leq
\frac{2}{\pi} \left| \sin \frac{2n-1}{2}x\right| 
\left( \frac{1}{n}\sum\limits_{k=n}^{\infty}k\psi(k) \right)
E_{n}(f^{\psi}_{\beta})_{L_{1}}.
\end{align}
Дійсно, \eqref{interpolation_generalCase_C^psi_Inequality_modified} випливає з \eqref{interpolation_generalCase_C^psi_Inequality} і того, що в силу \eqref{Lemma_SumInterpolationPart_Inequality}
\begin{align*}
&\sum\limits_{k=1}^{\infty} \sum\limits_{\nu=(2k+1)n-k}^{\infty}
\psi(\nu)
\leq 
\frac{1}{n} \sum\limits_{k=1}^{\infty}k\psi(k+n)
=\frac{1}{n} \sum\limits_{k=0}^{\infty}k\psi(k+n)
\notag \\
=&\frac{1}{n} \sum\limits_{k=n}^{\infty}(k-n)\psi(k)
=\frac{1}{n} \sum\limits_{k=n}^{\infty}k\psi(k)- \sum\limits_{k=n}^{\infty} \psi(k),
\end{align*}
а, отже,
\begin{equation}\label{ineqAddit}
\sum\limits_{k=0}^{\infty} \sum\limits_{\nu=(2k+1)n-k}^{\infty}
\psi(\nu)
\leq  \frac{1}{n} \sum\limits_{k=n}^{\infty}k\psi(k).
\end{equation}

Нерівність \eqref{interpolation_generalCase_C^psi_Inequality_modified}, будучи грубшою оцінкою у порівнянні із \eqref{interpolation_generalCase_C^psi_Inequality}, в деяких випадках є більш зручною для використання, залишаючись асимптотично непокращуваною при $n\rightarrow\infty$.

Іноді буває корисним формулювати нерівності типу Лебега  в термінах $(\psi, \beta)$--інтегралів $\mathcal{J}^{\psi}_{\beta}\varphi$ від довільної сумовної функції $\varphi$. Теорема~\ref{Interpolation_2023_theorem_GeneralCase} в цьому випадку може бути сформульована наступним чином.

%\ref{Interpolation_2023_theorem_GeneralCase}

{\bf Теорема $1'$.}
{\it 
Нехай $\sum\limits_{k=1}^{\infty}k\psi(k)<\infty$, $\psi(k)\geq 0$, $k=1,2,...$, $\beta\in\mathbb{R}$ i $n\in\mathbb{N}$.  Тоді, для всіх $x\in\mathbb{R}$ і довільної  функції
 $\varphi\in L_{1}$  має місце нерівність
\begin{align}\label{interpolation_generalCase_C^psi_Inequality_psiBetaIntegral}
|\tilde{\rho}_{n}(\mathcal{J}^{\psi}_{\beta}\varphi;x)|
\leq
\frac{2}{\pi} \left| \sin \frac{2n-1}{2}x\right| 
\left( \sum\limits_{k=0}^{\infty}\sum\limits_{\nu=(2k+1)n-k}^{\infty} \psi(\nu) \right)
E_{n}(\varphi)_{L_{1}}.
\end{align}
%where $F(a,b;c;d)$ is Gauss hypergeometric function.

Крім того, для довільної функції $\varphi\in L_{1}$  можна вказати функцію $\Phi(\cdot)=  \Phi(\varphi; n; \psi; \beta)$ із $L_{1}^{0}$ таку, що $E_{n}(\varphi)_{L_1}=E_{n}(\Phi)_{L_1}$  і для якої виконується наступна рівність:
\begin{align}\label{interpolation_generalCase_C^psi_Equality_psiBetaIntegral}
|\tilde{\rho}_{n}( \Phi;x)|
=
 \frac{2}{\pi} \left| \sin\frac{2n-1}{2}x \right|
 \left( 
\sum\limits_{k=n}^{\infty}\psi(k)+
\frac{ \xi }{n}\sum\limits_{k=1}^{\infty} k\psi(k+n) \right)
E_{n}(  \varphi)_{L_{1}},
\end{align}
де  величина $\xi=\xi(\varphi;n;\psi;\beta;x)$  є такою, що $-\left(1+2\pi \right)\leq \xi \leq 1$.
}

\begin{theorem}\label{sup_Interpolation_2023_theorem_GeneralCase}
Нехай $\sum\limits_{k=1}^{\infty}k\psi(k)<\infty$, $\psi(k)\geq 0$, $k=1,2,...$, $\beta\in\mathbb{R}$ i $n\in\mathbb{N}$.  Тоді, для всіх $x\in\mathbb{R}$ виконується рівність
\begin{align}\label{sup_interpolation_generalCase_C^psi_Equality}
\tilde{\mathcal{E}}_{n}(C^{\psi}_{\beta,1};x)=
\frac{2}{\pi}\left| \sin\frac{2n-1}{2}x \right|
\left( 
\sum\limits_{k=n}^{\infty}\psi(k)+
\frac{\Theta}{n} \sum\limits_{k=1}^{\infty} k\psi(k+n)  \right).
\end{align}
В \eqref{sup_interpolation_generalCase_C^psi_Equality} величина $\Theta=\Theta(n;\psi;\beta;x)$  є такою, що $-\left(1+\pi \right)\leq \Theta \leq 1$.
\end{theorem}

\begin{proof}[Доведення теореми~\ref{sup_Interpolation_2023_theorem_GeneralCase}]

Будемо відштовхуватись від інтегрального зображення \eqref{IntegrRepr}, у якому $f\in C^{\psi}_{\beta,1}$.  Розглянувши точні верхні межі модулів обох частин рівності  \eqref{IntegrRepr} при $t_{n-1} \equiv 0$ по класу $C^{\psi}_{\beta,1}$ та врахувавши  інваріантність множини 
\begin{equation*}
B_{1}^{0}= \left\{\varphi\in L_{1}: \ \|\varphi\|_{1}\leq 1, \ \int\limits_{-\pi}^{\pi}\varphi(t)dt=0 \right\}, 
\end{equation*}
відносно зсуву аргументу, %при всіх 
будемо мати
\begin{align}\label{Equation_InProof_Th_Interp1_particular}
&\tilde{\mathcal{E}}_{n}(C^{\psi}_{\beta, 1};x)
=
\sup\limits_{f\in C^{\psi}_{\beta, 1}}   |\tilde{\rho}_{n}(f;x)|
\notag \\
=&
2  \left|\sin \frac{2n-1}{2}x \right| 
\left(
\sup\limits_{\varphi \in B_{1}^{0}} \frac{1}{\pi}\int\limits_{-\pi}^{\pi} \varphi(t) \sum\limits_{k=n}^{\infty}\psi(k)\cos(kt+\gamma_{n}) dt+ \frac{\xi_{3}}{\pi} \|r_{n}(\cdot)\|_{C}
\right),
\end{align}
де $r_{n}(t)$ i $\gamma_{n}$  означені рівностями  \eqref{rn} i \eqref{gamma_n}  відповідно, а для величини $\xi_{3}=\xi_{3}(n; \psi; \beta,x)$ виконується нерівність $|\xi_{3}|\leq 1$.

Із співвідношень двоїстості (див., наприклад, \cite[с. 27]{Korn}) маємо
\begin{align}\label{Equation_InProof_Th_Interp2_Particular}
\sup\limits_{\varphi \in B_{1}^{0}} \frac{1}{\pi}\int\limits_{-\pi}^{\pi} \varphi(t) \sum\limits_{k=n}^{\infty}\psi(k)\cos(kt+\gamma_{n}) dt=
\frac{1}{\pi}\inf\limits_{\lambda \in\mathbb{R}} \left\| \sum\limits_{k=n}^{\infty}\psi(k)  \cos(kt+\gamma_{n}) -\lambda \right\|_{\infty}.
\end{align}

З леми~2.1 роботи \cite{SerdyukStepanyuk_UMJ_1_2023} випливає, що при  $\psi(k)\geq 0$, $\sum\limits_{k=1}^{\infty}k\psi(k) <\infty$, при всіх $\beta\in\mathbb{R}$ i $n\in\mathbb{N}$   виконується формула
\begin{equation}\label{Norm2_particular}
\inf\limits_{\lambda\in\mathbb{R}}  \left \|  \sum\limits_{k=n}^{\infty}\psi(k)\cos (kt+\gamma_{n})  -\lambda \right\|_{\infty}=
\sum\limits_{k=n}^{\infty}\psi(k)+
 \frac{\Theta_{1}\pi}{n} \sum\limits_{k=1}^{\infty}k\psi(k+n),
\end{equation}
в якій  для  величини $\Theta_{1}=\Theta_{1}(n, \beta,\psi,x)$ виконується двостороння оцінка
$-1\leq \Theta_{1}\leq 0$.

Об'єднавши формули \eqref{Equation_InProof_Th_Interp1_particular}--\eqref{Norm2_particular}  i врахувавши \eqref{rn} та \eqref{Lemma_SumInterpolationPart_Inequality}, отримуємо \eqref{sup_interpolation_generalCase_C^psi_Equality}.
Теорему~\ref{sup_Interpolation_2023_theorem_GeneralCase} доведено.
\end{proof}

\section{Нерівності Лебега для інтерполяційних поліномів на класах згорток  нескінченно диференційовних, аналітичних та цілих  функцій }\label{Section_corrolary_infinitelyDifferentiable}

В даному  підрозділі роботи ми наведемо приклади функціональних множин $L^{\psi}_{\beta}L_{1}$ для яких оцінки \eqref{interpolation_generalCase_C^psi_Inequality} i \eqref{interpolation_generalCase_C^psi_Equality} з теореми~\ref{Interpolation_2023_theorem_GeneralCase} є асимптотично точними.

\subsection{Наслідки з теореми~\ref{Interpolation_2023_theorem_GeneralCase} для множин нескінченно диференційовних функцій}\label{Subsection_Section_corrolary_infinitelyDifferentiable}

В даному пункті будемо розглядати випадок, коли послідовності $\psi(k)$, що породжують множини $C^{\psi}_{\beta}L_{1}$  є звуженням на множину натуральних чисел деяких додатних неперервних опуклих донизу функцій  $\psi(t)$ неперервного аргументу $t\geq 1$, що прямують до нуля при $t\rightarrow\infty$.
Множину всіх таких функцій $\psi$ позначають через $\mathfrak{M}$:
\begin{equation}\label{Mathfrak_M}
\mathfrak{M}\!=\! \left\{\psi\!\in \! C[1,\infty)\!: \psi(t)\!>\!0, \psi(t_{1}-2\psi((t_{1}+t_{2})/2) +\psi(t_{2})\geq 0 \ \forall t_{1}, t_{2} \in[1,\infty), \  \lim\limits_{t\rightarrow\infty}\psi(t)\!=\!0 \right\}.
\end{equation}

Наслідуючи О.І. Степанця (див., наприклад, \cite[с. 160]{Stepanets1}), кожній функції $\psi\in\mathfrak{M}$ поставимо у відповідність характеристику
\begin{equation*}
\mu(t)=\mu(\psi;t)=\frac{t}{\eta(t)-t},
\end{equation*}
де 
$\eta(t)=\eta(\psi;t)=\psi^{-1}\left( \frac{1}{2}\psi(t)\right)$,
$\psi^{-1}$ --- обернена до $\psi$ функція, і покладемо
\begin{equation*}
\mathfrak{M}_{\infty}^{+}= \left\{\psi\in \mathfrak{M}: \ \mu(t)\uparrow \infty, \ \ t\rightarrow\infty \right\}.
\end{equation*}

Через $\mathfrak{M}^{\alpha}$ позначимо підмножину всіх функцій 
$\psi\in \mathfrak{M}$, для яких величина
 \begin{equation}\label{psi_alpha}
\alpha(t)=\alpha(\psi;t):=\frac{\psi(t)}{t|\psi'(t)|}, \ \ \psi'(t):=\psi'(t+0),
\end{equation}
спадає до нуля при $t\rightarrow\infty$:
\begin{equation}\label{Mathfrak_M_alpha}
\mathfrak{M}^{\alpha}= \left\{\psi\in \mathfrak{M}:   \ \ \alpha(\psi;t)\downarrow0, \ \ t\rightarrow\infty \right\}.
\end{equation}

Згідно  з теоремою 2 роботи \cite{Stepanets_Serdyuk_Shydlich2007}, випливає, що твердження про існування послідовності  $\psi\in \mathfrak{M}^{\alpha}$ (або $\psi\in\mathfrak{M}_{\infty}^{+}$), такої, що для функції $f$ вірне  включення $f\in C^{\psi}_{\beta}L_{1}$  при будь-якому $\beta\in\mathbb{R}$, еквівалентне твердженню про включення $f\in D^{\infty}$, де $D^{\infty}$ --- множина всіх нескінченно диференційовних $2\pi$-періодичних дійснозначних функцій.
А отже, множини $C^{\psi}_{\beta}L_{1}$  при $\psi\in \mathfrak{M}^{\alpha}$ (або $\psi\in\mathfrak{M}_{\infty}^{+}$),  є множинами нескінченно диференційовних періодичних функцій.  В  роботі  \cite{Stepanets_Serdyuk_Shydlich2008} було доведено включення
\begin{equation}\label{Mathfrak_M_alpha}
\mathfrak{M}_{\infty}^{+} \subset \mathfrak{M}^{\alpha}\subset \mathfrak{M}^{\infty}= \left\{\psi\in \mathfrak{M}:  \ \forall r>0 \ \ \lim\limits_{t\rightarrow\infty} t^{r}\psi(t)=0 \right\},
\end{equation}
яке означає, що функції $\psi$ із $\mathfrak{M}^{\alpha}$ спадають до нуля швидше за довільну степеневу функцію.

 Як показано в \cite[с. 166]{Stepanets1} для довільної функції $\psi$ із $\mathfrak{M}_{\infty}^{+}$ має місце порядкова рівність
\begin{equation}\label{OrderEstimates_theta_and_lambda}
\eta(t)-t \asymp \lambda(t), \ t\geq 1,
 \end{equation}
де $\lambda(t)$ --- характеристика вигляду
\begin{equation}\label{psi_lambda}
\lambda(t)=\lambda(\psi;t):=\frac{\psi(t)}{|\psi'(t)|}.
\end{equation}

В даному пункті нас буде цікавити випадок, коли для $\psi\in \mathfrak{M}^{\alpha}$, a характеристика $\lambda(t)$ монотонно зростає до нечскінченності при $t\rightarrow \infty$. При цих обмеженнях функції $\psi(t)$ спадають до нуля швидше за довільну степеневу функцію, але повільніше за будь-яку геометричну прогресію.

\begin{theorem}\label{theorem2_M}
Нехай  $\psi\in\mathfrak{M}$ i характеристики $\alpha(t)$ вигляду \eqref{psi_alpha} та  $\lambda(t)$ вигляду \eqref{psi_lambda} задовольняють умови
\begin{equation}\label{alphaTo_0}
\alpha(t)\downarrow0,  \ \ t\rightarrow\infty,
\end{equation}
\begin{equation}\label{lambdaTo_infty}
\lambda(t)\uparrow \infty, \ \ t\rightarrow\infty.
\end{equation}
Тоді,  довільної  функції
 $f\in C^{\psi}_{\beta}L_{1}$, $\beta\in\mathbb{R}$,  в кожній точці $x\in\mathbb{R}$  при всіх $n\in \mathbb{N}$ таких, що
\begin{equation}\label{alpha_1/4}
\alpha(n)\leq \frac{1}{4},
\end{equation}
має місце нерівність 
\begin{equation}\label{interpolation_ParticularCase_C^psi_Inequality}
|\tilde{\rho}_{n}(f;x)|
\leq
\frac{2}{\pi} \left| \sin \frac{2n-1}{2}x\right| 
\psi(n) \lambda(n)
\left( 1+ 4\alpha(n) +\frac{1}{\lambda(n)}
\right) E_{n}(f^{\psi}_{\beta})_{L_{1}}.
 \end{equation}

 Крім того, для довільної функції $f\in C^{\psi}_{\beta}L_{1}$ можна вказати функцію $\mathcal{F}(\cdot)=\mathcal{F}(f;n;x, \cdot)$ таку, що $E_{n}(\mathcal{F}^{\psi}_{\beta})_{L_1}=E_{n}(f^{\psi}_{\beta})_{L_1}$  і виконується наступна рівність:
\begin{align}\label{interpolation_particularCase_C^psi_Equality}
|\tilde{\rho}_{n}(\mathcal{F}; x)|
=
\frac{2}{\pi} \left| \sin \frac{2n-1}{2}x\right| 
\psi(n) \lambda(n)
 \left( 1+ \xi_{3}\alpha(n)+ \frac{\xi_{4}}{\lambda(n)}
\right)
E_{n}(f^{\psi}_{\beta})_{L_{1}},
\end{align}
де  $-4(1+2\pi)\leq \xi_{3}<\frac{8}{3}(1+\pi)$,  $-(1+2\pi)\leq \xi_{4}\leq 2(1+\pi)$.
\end{theorem}

 \begin{proof}[Доведення теореми~\ref{theorem2_M}] 
 Умова \eqref{alphaTo_0} для $\psi\in\mathfrak{M}$ гарантує виконання всіх умов теореми~\ref{Interpolation_2023_theorem_GeneralCase}. Покажемо спочатку, що \eqref{interpolation_ParticularCase_C^psi_Inequality} випливає із \eqref{interpolation_generalCase_C^psi_Inequality_modified}.
 
 У роботі \cite[с. 561]{SerdyukStepanyuk_UMJ_1_2023} було доведено, що за виконання умов  \eqref{alphaTo_0}, \eqref{lambdaTo_infty} i \eqref{alpha_1/4} має місце наступна оцінка
 \begin{align}\label{form6}
\frac{1}{n}\sum\limits_{k=n}^{\infty}k\psi(k)
=
 \psi(n)  \lambda(n) \left(1+\Theta_{2}\alpha(n)+
\frac{\Theta_{3}}{\lambda(n)} \right).
 \end{align}
де  $0\leq \Theta_{2} \leq 4$,  $0\leq \Theta_{3}\leq 1$.
 З \eqref{interpolation_generalCase_C^psi_Inequality_modified} i \eqref{form6} випливає \eqref{interpolation_ParticularCase_C^psi_Inequality}.
 
 Щоб переконатись у справедливості рівності  \eqref{interpolation_particularCase_C^psi_Equality}  застосуємо формулу \eqref{interpolation_generalCase_C^psi_Equality} теореми~\ref{Interpolation_2023_theorem_GeneralCase} і використаємо формули (85) і (91) роботи \cite{SerdyukStepanyuk_UMJ_1_2023}, згідно з якими для довільної $\psi\in\mathfrak{M}$, що підпорядкована умовам \eqref{alphaTo_0} i \eqref{lambdaTo_infty} при всіх $n\in\mathbb{N}$, що задовольняють умову \eqref{alpha_1/4} мають місце оцінки
 \begin{equation}\label{Sum_psi(k)}
 \sum\limits_{k=n}^{\infty}\psi(k)= \psi(n)\lambda(n) \left( 1+ \Theta_{4}\alpha(n)+ \frac{\Theta_{5}}{\lambda(n)} \right), \ \ 0\leq \Theta_{4} \leq \frac{4}{3}, \ 0\leq \Theta_{5}\leq 1,
\end{equation}
\begin{align}\label{Sum_kpsi(k+n)}
\frac{1}{n}\sum\limits_{k=1}^{\infty}k\psi(k+n)
=
 \psi(n)  \lambda(n) \left( \Theta_{6}\alpha(n)+ \frac{\Theta_{7}}{\lambda(n)} \right), \ \ -\frac{4}{3}\leq \Theta_{6} \leq 4, \ -1\leq \Theta_{7}\leq 1.
 \end{align}

 З урахуванням \eqref{Sum_psi(k)} i \eqref{Sum_kpsi(k+n)} рівність \eqref{interpolation_generalCase_C^psi_Equality} теореми~\ref{Interpolation_2023_theorem_GeneralCase} можна записати у вигляді
 
 \begin{align}\label{interpolation_Particular CaseCase_C^psi_form1}
|\tilde{\rho}_{n}(\mathcal{F}; x)|
=&
\frac{2}{\pi}\left| \sin\frac{2n-1}{2}x \right|
 \left( 
\sum\limits_{k=n}^{\infty}\psi(k)+
 \frac{\xi }{n}\sum\limits_{k=1}^{\infty} k\psi(k+n)\right)
E_{n}(f^{\psi}_{\beta})_{L_{1}} \notag\\
=&
\frac{2}{\pi}\left| \sin\frac{2n-1}{2}x \right|
\psi(n)\lambda(n) \left( 1+
(\Theta_{4}+\xi\Theta_{6})\alpha(n) + (\Theta_{5}+\xi\Theta_{7}) \frac{1}{\lambda(n)}
\right)
E_{n}(f^{\psi}_{\beta})_{L_{1}}.
\end{align}
Із \eqref{interpolation_Particular CaseCase_C^psi_form1} випливає \eqref{interpolation_particularCase_C^psi_Equality}.
Теорему~\ref{theorem2_M} доведено.
 \end{proof}
 
Важливими множинами $C^{\psi}_{\beta}L_{1}$ для яких функції $\psi \, (\psi\in \mathfrak{M})$ задовольняють усі умови теореми~\ref{theorem2_M} є множини узагальнених інтегралів Пуассона $C^{\alpha,r}_{\beta}L_{1}$, $\alpha>0$, $\beta\in\mathbb{R}$, $r\in(0,1)$.
В цьому випадку для $\psi(t)=e^{-\alpha t^{r}}$ при довільних $t\geq 1$
\begin{equation}\label{lambda_alpha_generalizedPoisson_Integtals}
\lambda(t)=\frac{t^{1-r}}{\alpha r}, \ \ \alpha(t)=\frac{1}{\alpha 
r t^{r}},
 \end{equation}
а отже $\lambda(t)\uparrow\infty$, $\alpha(t)\downarrow 0$, $t\rightarrow\infty$. З урахуванням \eqref{lambda_alpha_generalizedPoisson_Integtals} умова \eqref{alpha_1/4} може бути записана у вигляді $n\geq \left(\frac{4}{\alpha r}\right)^{\frac{1}{r}}$.

\begin{corollary}\label{theorem_p=1}
Нехай $0<r<1$, $\alpha>0$, $\beta\in\mathbb{R}$. 
 Тоді, для довільної функції
 $f\in C^{\alpha,r}_{\beta}L_{1}$, $\beta\in\mathbb{R}$,  в кожній точці $x\in\mathbb{R}$ при всіх $n\in\mathbb{N}$ таких, що $n\geq  \left(\frac{4}{\alpha r}\right)^{\frac{1}{r}} $ виконується нерівність
 \begin{equation}\label{Theorem_Ineq1_p=1}
|\tilde{\rho}_{n}(f;x)|
\leq
 e^{-\alpha n^{r}}n^{1-r}
\left|\sin \frac{2n-1}{2}x \right| 
\Big(
\frac{2}{\pi\alpha r}+\mathcal{O}(1)\Big(\frac{1}{(\alpha r)^{2}}\frac{1}{n^{r}}+\frac{1}{n^{1-r}}\Big)\Big)E_{n}(f^{\alpha,r}_{\beta})_{L_{1}}.
 \end{equation}
 
 Крім того, для довільної функції  $f\in C^{\alpha,r}_{\beta}L_{1}$ можна вказати функцію ${\mathcal{F}(\cdot)=\mathcal{F}(f;n; x, \cdot)}$ з множини $C^{\alpha,r}_{\beta}L_{1}$ таку, що
  $E_{n}(\mathcal{F}^{\alpha,r}_{\beta})_{L_{1}}=E_{n}(f^{\alpha,r}_{\beta})_{L_{1}}$ 
і при  $n\geq n_{*}(\alpha,r,1)$  має місце рівність
 \begin{equation}\label{Theorem_Eq_p=1}
|\tilde{\rho}_{n}(\mathcal{F}; x)|
=
 e^{-\alpha n^{r}}n^{1-r}
\left|\sin \frac{2n-1}{2}x \right| 
\Big(
\frac{2}{\pi\alpha r}+\mathcal{O}(1)\Big(\frac{1}{(\alpha r)^{2}}\frac{1}{n^{r}}+\frac{1}{n^{1-r}}\Big)\Big)E_{n}(f^{\alpha,r}_{\beta})_{L_{1}}.
 \end{equation}  
В (\ref{Theorem_Ineq1_p=1})  і  (\ref{Theorem_Eq_p=1}) величини $\mathcal{O}(1)$ рівномірно обмежені по всіх розглядуваних параметрах.
\end{corollary}

Зауважимо також, що формули (\ref{Theorem_Ineq1_p=1})  і  (\ref{Theorem_Eq_p=1})  випливають також з  теореми 2.3 роботи авторів \cite{SerdyukStepanyuk2023_UMJ_No7}.

Можна навести цілий ряд прикладів функцй $\psi$ із $\mathfrak{M}$, для яких виконуються умови \eqref{alphaTo_0} i \eqref{lambdaTo_infty}  теореми~\ref{theorem2_M}. Зокрема: 

\begin{equation}\label{psi_1}
\psi(t)=(t+2)^{- \ln\ln (t+2)}, \ \ t\geq 1,
\end{equation}
\begin{equation}\label{psi_2}
\psi(t)=e^{- \ln^{2} (t+1)}, \ \ t\geq 1,
\end{equation}
\begin{equation}\label{psi_3}
\psi(t)=e^{- \frac{t+2}{\ln (t+2) }}, \ \ t\geq 1.
\end{equation}

Для усіх перелічених функцій $\psi$ теорема~\ref{theorem2_M} дозволяє записати асимптотично непокращувані нерівності типу Лебега для інтерполяційних поліномів Лагранжа на множинах $C^{\psi}_{\beta}L_{1}$, $\beta\in\mathbb{R}$.

\subsection{Наслідки з теореми~\ref{Interpolation_2023_theorem_GeneralCase} для множин аналітичних та цілих функцій}\label{Subsection_analytic_Section_corrolary_infinitelyDifferentiable}

В даному пункті будемо розглядати випадки, коли послідовності $\psi(k)$ спадають до нуля приблизно як геометричні прогресії або швидше них. Як показано в \cite[c. 32]{Step monog 1987},  \cite[c. 1697]{Stepanets_Serdyuk_Shydlich2008}, якщо при деяких $K$ i $\alpha$ $\psi(k)$ задовольняє умову
\begin{equation}\label{psi_4Inequality}
\psi(k)\leq Ke^{-\alpha k}, \ \ k\in\mathbb{N}, \ \ \alpha>0, \ \ K>0,
\end{equation}
то множини $ C^{\alpha,r}_{\beta}L_{1}$ складаються  з аналітичних функцій, які регулярно продовжуються у смугу $|\mathrm{Im} \, z |<\alpha$. Типовим прикладом функції $\psi(k)$, що задовольняє умову \eqref{psi_4Inequality} може слугувати функція
\begin{equation}\label{psi_4}
\psi(k)= e^{-\alpha k},  \ \ \alpha>0,
\end{equation}
яка породжує множину інтегралів Пуассона $C^{\alpha,1}_{\beta}L_{1}$.

Оскільки для функції вигляду \eqref{psi_4} виконуються рівності
\begin{align}\label{Sums_generalizedPoissonIntegrals_1}
&
\sum\limits_{k=0}^{\infty} \sum\limits_{\nu=(2k+1)n-k}^{\infty} \psi(\nu)
=
\sum\limits_{k=0}^{\infty} \sum\limits_{\nu=(2k+1)n-k}^{\infty} e^{-\alpha\nu}
=\sum\limits_{k=0}^{\infty}  \frac{e^{-\alpha (2k+1)n-k}}{1-e^{-\alpha}}
\notag \\
=&
 \frac{e^{-\alpha n}}{1-e^{-\alpha}} \sum\limits_{k=0}^{\infty} e^{-\alpha (2n-1)k}
 = \frac{e^{-\alpha n}}{(1-e^{-\alpha}) (1-e^{-\alpha (2n-1)}) };
\end{align}

\begin{align}\label{Sums_generalizedPoissonIntegrals_2}
\frac{1}{n}\sum\limits_{k=1}^{\infty} k\psi(k+n)
=\frac{1}{n}\sum\limits_{k=1}^{\infty} k e^{-\alpha (k+n)}=
\frac{e^{-\alpha n}}{n}  \sum\limits_{k=1}^{\infty} k e^{-\alpha k}
 = \frac{e^{-\alpha (n+1)}}{n(1-e^{-\alpha})^{2} },
\end{align}
то із теореми~\ref{Interpolation_2023_theorem_GeneralCase} одержуємо твердження.

\begin{theorem}\label{theorem_p=1_r=1}
Нехай $\alpha>0$ i $\beta\in\mathbb{R}$. 
 Тоді, для всіх $x\in\mathbb{R}$,  $n\in\mathbb{N}$ і довільної функції 
$f\in C^{\alpha,1}_{\beta}$
 має місце  нерівність
 \begin{equation}\label{Theorem_Ineq1_p=1_r=1}
|\tilde{\rho}_{n}(f;x)|
\leq
 \frac{2}{\pi} 
 \left|\sin \frac{2n-1}{2}x \right| 
\frac{e^{-\alpha n}}{(1-e^{-\alpha}) (1-e^{-\alpha (2n-1)}) }
E_{n}(f^{\alpha,1}_{\beta})_{L_{1}}.
 \end{equation}
 
 Крім того, для довільної функції  $f\in C^{\alpha,1}_{\beta}L_{1}$ можна вказати функцію ${\mathcal{F}(\cdot)=\mathcal{F}(f;n; x, \cdot)}$ з множини $C^{\alpha,1}_{\beta}L_{1}$ таку, що
  $E_{n}(\mathcal{F}^{\alpha,1}_{\beta})_{L_{1}}=E_{n}(f^{\alpha,1}_{\beta})_{L_{1}}$ 
і для якої виконується  рівність
 \begin{equation}\label{Theorem_Eq_p=1_r=1}
|\tilde{\rho}_{n}(\mathcal{F}; x)|
=
 \frac{2}{\pi}  \left|\sin \frac{2n-1}{2}x \right| 
 e^{-\alpha n}
\left(  \frac{1}{1-e^{-\alpha}}  + \frac{\xi}{n}  \frac{e^{-\alpha}}{(1-e^{-\alpha})^{2} } \right)
E_{n}(f^{\alpha,1}_{\beta})_{L_{1}},
 \end{equation}  
де для величини $\xi=\xi(f,n,\alpha, \beta;x)$ виконується оцінка $-(1+2\pi)\leq \xi\leq 1$.
\end{theorem}

Важливим частинним випадком функцій $\psi$, підпорядкованих умові \eqref{psi_4Inequality} є $\psi(k)$, що задовольняють умову Даламбера $\mathcal{D}_{q}$, $q\in[0,1)$:
\begin{equation}\label{DalamberCondition}
\lim\limits_{k\rightarrow\infty}\frac{\psi(k+1)}{\psi(k)}=q, \ \ \ \psi(k)>0.
\end{equation}

Якщо $\psi(k)$ задовольняє умову \eqref{DalamberCondition} при деякому $q\in[0,1)$, то будемо записувати, що $\psi\in \mathcal{D}_{q}$.

\begin{corollary}\label{Interpolation_2023_theorem_Dq}
Нехай $\psi\in \mathcal{D}_{q}$, $q\in(0,1)$,  $\beta\in\mathbb{R}$.
Тоді  для всіх $x\in\mathbb{R}$ і довільної  функції
 $f\in C^{\psi}_{\beta}L_{1}$  при  всіх номерах $n$ таких, що
\begin{equation}\label{Number_n_Dq}
\frac{1}{n}+\varepsilon_{n}<\frac{1-q}{2},
\end{equation}
де
\begin{equation}\label{Varepselon_n_Dq}
\varepsilon_{n}:=\sup\limits_{k\geq n} \left| \frac{\psi(k+1)}{\psi(k)}-q\right|,
\end{equation}
  має місце нерівність
\begin{align}\label{interpolation_Dq_Inequality}
|\tilde{\rho}_{n}(f;x)|
\leq
\left| \sin \frac{2n-1}{2}x\right| 
\psi(n) \left(
\frac{2}{\pi(1-q)}+\mathcal{O}(1)\Big(\frac{q}{n(1-q)^{2}}+\frac{\varepsilon_{n}}{(1-q)^{2}}\Big)\right)
E_{n}(f^{\psi}_{\beta})_{L_{1}}.
\end{align}
%where $F(a,b;c;d)$ is Gauss hypergeometric function.

Крім того, для довільної функції $f\in C^{\psi}_{\beta}L_{1}$ можна вказати функцію $\mathcal{F}(\cdot)=\mathcal{F}(f;n;x, \cdot)$ таку, що $E_{n}(\mathcal{F}^{\psi}_{\beta})_{L_1}=E_{n}(f^{\psi}_{\beta})_{L_1}$  і для якої виконується наступна рівність:
\begin{align}\label{interpolation_Dq_Equality}
|\tilde{\rho}_{n}(\mathcal{F}; x)|
=
 \left| \sin \frac{2n-1}{2}x\right| 
\psi(n) \left(
\frac{2}{\pi(1-q)}+\mathcal{O}(1)\Big(\frac{q}{n(1-q)^{2}}+\frac{\varepsilon_{n}}{(1-q)^{2}}\Big)\right)
E_{n}(f^{\psi}_{\beta})_{L_{1}}.
\end{align}
\end{corollary}

\begin{proof} 
В  роботі \cite[c. 557]{SerdyukStepanyuk_UMJ_1_2023} було показано, що коли $\psi\in \mathcal{ D}_{q}$, $q\in(0,1)$, то для всіх номерів $n$, що задовольняють умову \eqref{Number_n_Dq} справджуються рівності
 
 \begin{align}\label{Formulas_Corrolary}
&  \sum\limits_{k=n}^{\infty}\psi(k) +
\frac{\mathcal{O}(1)}{n}\sum\limits_{k=1}^{\infty} k\psi(k+n) \notag \\
  =&
  \psi(n)
\left( \frac{1}{1-q} +\mathcal{O}(1)
\left( \frac{\varepsilon_{n}}{(1-q)^{2}} + \frac{q}{n(1-q)^{2}}  \right)
\right).
\end{align}
Використавши нерівність \eqref{Lemma_SumInterpolationPart_Inequality} та оцінку \eqref{Formulas_Corrolary}, із формул \eqref{interpolation_generalCase_C^psi_Inequality} i \eqref{interpolation_generalCase_C^psi_Equality} теореми~\ref{Interpolation_2023_theorem_GeneralCase} одержуємо оцінки \eqref{interpolation_Dq_Inequality} i \eqref{interpolation_Dq_Equality} відповідно.
Наслідок \ref{Interpolation_2023_theorem_Dq} доведено.

\end{proof}

Важливими прикладами ядер $\Psi_{\beta}(t)$ вигляду \eqref{kernelPsi}
коефіцієнти $\psi(k)$ яких належать до множини ${\cal D}_q$,
$0<q<1$, є, зокрема,

$\bullet$ полігармонічні ядра Пуассона
\begin{equation}
\label{3Kern_PUAS_POLI}
P_{q,\beta}(l,t)=\sum\limits_{k=1}^{\infty}\psi_l(k)\cos\left(kt-\frac{\beta\pi}{2}\right),
l\in\mathbb{N}, \ \ \beta\in\mathbb{R},
\end{equation}
де
\begin{equation}
\label{3koef_KERN_P_POLI}
\psi_l(k)=q^k\left(1+\sum\limits_{j=1}^{l-1}\frac{(1-q^2)^j}{j!2^j}
\prod\limits_{\nu=0}^{j-1}(k+2\nu)\right), 0<q<1,
\end{equation}

$\bullet$  ядра аналітичних функцій
\begin{equation}
\label{3Kern_PUAS_TEPL} {\cal
P}_{q,\beta}(t)=\sum\limits_{k=1}^{\infty}\frac{2}{q^k+q^{-k}}\cos\left(kt-\frac{\beta\pi}{2}\right), \ \ 
0<q<1, \ \ \beta\in\mathbb{R},
\end{equation}

$\bullet$  ядра Неймана
\begin{equation}
\label{3Kern_Neim}
N_{q,\beta}(t)=\sum\limits_{k=1}^{\infty}\frac{q^k}{k}\cos\left(kt-\frac{\beta\pi}{2}\right), \ \ 
0<q<1, \ \  \beta\in\mathbb{R},
\end{equation}
та інші.

  Якщо $C^{\psi}_{\beta}L_{1}$  є множинами згорток з усіма наведеними ядрами, то наслідок~\ref{Interpolation_2023_theorem_Dq} дозволяє записати для них асимптотично непокращувані інтерполяційні нерівності типу Лебега. Зазначимо, що формули \eqref{interpolation_Dq_Inequality} i \eqref{interpolation_Dq_Equality} вперше були одержані в роботі \cite{Serdyuk2012}. В ній же були одержані асимптотично непокращувані нерівності типу Лебега для інтерполяційних поліномів на класах згорток $C^{\psi}_{\beta}L_{p}$, $1\leq p\leq\infty$, породжених ядрами \eqref{3Kern_PUAS_POLI} i \eqref{3Kern_Neim}, а також усіма ядрами $\Psi_{\beta}$ вигляду \eqref{kernelPsi}, де $\psi\in \mathcal{D}_{q}$, $q\in(0,1)$.

Розглянемо функцію $\psi(k)$ вигляду 
\begin{equation}\label{psi_even_oddCase}
\psi(k) = 
\begin{cases}
 q_{1}^{k}, & \text{якщо }
 k=2m-1, \ m\in\mathbb{N} \\
q_{2}^{k}, & \text{якщо }  k=2m,  m\in\mathbb{N},
  \end{cases}
\end{equation}
де числа $q_{1}$, $q_{2}$, такі, що
\begin{equation}\label{q1q2}
1>q_{1}>q_{2}>0.
\end{equation}

Зрозуміло, що функції $\psi$ вигляду \eqref{psi_even_oddCase} задовольняють умову \eqref{psi_4Inequality} при $\alpha>\ln \frac{1}{q_{1}}$, $K=1$, але водночас, ні при яких $q\in(0,1)$ не задовольняють ні умову Даламбера \eqref{DalamberCondition}, ні більш загальну, ніж $\mathcal{D}_{q}$ умову Коші $\mathcal{C}_{q}$:
\begin{equation}\label{CauschiConditiion}
\lim\limits_{k\rightarrow \infty} \sqrt[k]{\psi(k)} = q, \ \ q\in[0,1).
\end{equation}

В той же час теорема~\ref{Interpolation_2023_theorem_GeneralCase} дозволяє записати асимптотично непокращувані нерівності Лебега для інтерполяційних поліномів на множинах $C^{\psi}_{\beta}L_{1}$ для $\psi$ вигляду \eqref{psi_even_oddCase}. Для цього досить переконатись у справедливості граничного співвідношення \eqref{LimRelation}.

Неважко пересвідчитись, що для $\psi$ вигляду \eqref{psi_even_oddCase}
\begin{equation}\label{Sum_psi(k)_psi_even_oddCase}
\sum\limits_{k=n}^{\infty}\psi(k) = 
\begin{cases}
\frac{ q_{2}^{n}}{1- q_{2}^{2}}+  \frac{ q_{1}^{n+1}}{1- q_{1}^{2}}, & \text{якщо }
 n=2\ell, \ \ell\in\mathbb{N}, \\
\frac{ q_{1}^{n}}{1- q_{1}^{2}}+  \frac{ q_{2}^{n+1}}{1- q_{2}^{2}}, & \text{якщо }  n=2\ell-1, \ \ell\in\mathbb{N},
  \end{cases}
\end{equation}

\begin{equation}\label{1/nSum_kpsi(k+n)_psi_even_oddCase}
\frac{1}{n}\sum\limits_{k=1}^{\infty}k\psi(k+n) = 
\begin{cases}
\frac{1}{n} \left( \frac{2 q_{2}^{n+2}}{(1- q_{2}^{2})^{2}}+  \frac{ 2q_{1}^{n+3}}{(1- q_{1}^{2})^{2}} +  \frac{ q_{1}^{n+1}}{1- q_{1}^{2}} \right), & \text{якщо }
 n=2\ell, \ \ell\in\mathbb{N}, \\
\frac{1}{n} \left( \frac{ 2q_{1}^{n+2}}{(1- q_{1}^{2})^{2}}+  \frac{ 2q_{2}^{n+3}}{(1- q_{2}^{2})^{2}}+    \frac{ q_{2}^{n+1}}{1- q_{2}^{2}}  \right), & \text{якщо }  n=2\ell-1, \ \ell\in\mathbb{N}.
  \end{cases}
\end{equation}

Із \eqref{psi_alpha}, \eqref{Sum_psi(k)_psi_even_oddCase} i \eqref{1/nSum_kpsi(k+n)_psi_even_oddCase} одержуємо

\begin{equation}\label{Sum_psi(k)_psi_even_oddCase_LimitRelations}
\lim\limits_{n\rightarrow \infty}
\frac{\frac{1}{n}\sum\limits_{k=1}^{\infty}k\psi(k+n)}{\sum\limits_{k=n}^{\infty}\psi(k) }
= 
\begin{cases}
\lim\limits_{n\rightarrow\infty}   \frac{ \frac{1}{n} \left( \frac{ q_{1}^{n+2}}{1- q_{1}^{2}}+  \frac{ 2q_{1}^{n+3}}{(1- q_{1}^{2})^{2}}  \right) }{   \frac{ q_{1}^{n+1}}{1- q_{1}^{2}} }  =0, & \text{якщо }
 n=2\ell, \ \ell\in\mathbb{N}, \\
 \lim\limits_{n\rightarrow\infty} \frac{  \frac{1}{n}  \frac{ 2q_{1}^{n+2}}{(1- q_{1}^{2})^{2}} }{\frac{ q_{1}^{n}}{1- q_{1}^{2}}} =0, & \text{якщо }  n=2\ell-1, \ \ell\in\mathbb{N},
  \end{cases}
\end{equation}

Із  \eqref{Sum_psi(k)_psi_even_oddCase_LimitRelations} випливає \eqref{LimRelation}.

Розглянемо далі випадок, коли $\psi(k)$ задовольняють умову $\mathcal{D}_{0}$:
\begin{equation}\label{D0Condition}
\lim\limits_{k\rightarrow \infty}
\frac{\psi(k+1)}{\psi(k)}=0
\end{equation}

В цьому випадку множини $C^{\psi}_{\beta}L_{1}$ складаються із функцій регулярних в усій комплексній площині, тобто з цілих функцій (див., наприклад, \cite[с. 1703]{Stepanets_Serdyuk_Shydlich2008}).

Для $\psi\in \mathcal{D}_{0}$ із теореми~\ref{Interpolation_2023_theorem_GeneralCase} одержуємо наступне твердження.

\begin{corollary}\label{Interpolation_2023_theorem_D0}
Нехай  $\psi \in \mathcal{D}_{0}$,  $\beta\in\mathbb{R}$ i $n\in\mathbb{N}$.  Тоді, при всіх $x\in\mathbb{R}$ для  довільної  функції
 $f\in C^{\psi}_{\beta}L_{1}$  має місце нерівність
\begin{align}\label{interpolation_D0_Inequality}
|\tilde{\rho}_{n}(f;x)|
\leq
\frac{2}{\pi} \left| \sin \frac{2n-1}{2}x\right| 
\left( \psi(n)
+
\frac{1}{n} \sum\limits_{k=n+1}^{\infty} k\psi(k) \right)
E_{n}(f^{\psi}_{\beta})_{L_{1}}.
\end{align}
%where $F(a,b;c;d)$ is Gauss hypergeometric function.

Крім того, для довільної функції $f\in C^{\psi}_{\beta}L_{1}$, $\psi\in\mathcal{D}_{0}$ можна вказати функцію $\mathcal{F}(\cdot)=\mathcal{F}(f;n;x, \cdot)$ таку, що $E_{n}(\mathcal{F}^{\psi}_{\beta})_{L_1}=E_{n}(f^{\psi}_{\beta})_{L_1}$  і для якої виконується наступна рівність:
\begin{align}\label{interpolation_D0_Equality}
|\tilde{\rho}_{n}(\mathcal{F}; x)|
=
\frac{2}{\pi} \left| \sin \frac{2n-1}{2}x\right| 
\left( \psi(n)
+
\frac{\mathcal{O}(1)}{n} \sum\limits_{k=n+1}^{\infty} k\psi(k) \right)
E_{n}(f^{\psi}_{\beta})_{L_{1}}.
\end{align}
У формулі \eqref{interpolation_D0_Equality} величина $\mathcal{O}(1)$ рівномірно обмежена по всіх параметрах.
\end{corollary}

\begin{proof}
Нерівність  \eqref{interpolation_D0_Inequality} випливає безпосередньо із \eqref{interpolation_generalCase_C^psi_Inequality_modified}. Далі, оскільки
\begin{align}\label{interpolation_D0_Equality_ProofEq1}
&\sum\limits_{k=1}^{\infty}\psi(k)
+
\frac{\mathcal{O}(1)}{n}  \sum\limits_{k=1}^{\infty} k\psi(k+n)
=
\psi(n)+ \frac{\mathcal{O}(1)}{n} \left(  n  \sum\limits_{k=n+1}^{\infty} \psi(k)
+ \sum\limits_{k=1}^{\infty} k\psi(k+n)
\right) \notag \\
=&
\psi(n)+ \frac{\mathcal{O}(1)}{n} \left(  n  \sum\limits_{k=n+1}^{\infty} \psi(k)
+ \sum\limits_{k=n+1}^{\infty} (k-n)\psi(k)
\right)
=\psi(n)+ \frac{\mathcal{O}(1)}{n} \sum\limits_{k=n+1}^{\infty} k\psi(k),
\end{align}
то \eqref{interpolation_D0_Equality} є наслідком із \eqref{interpolation_generalCase_C^psi_Equality}. Той факт, що  \eqref{interpolation_D0_Equality} є асимптотичною рівністю при $n\rightarrow\infty$ випливає із формули

\begin{equation}\label{LimitRelations_Sum/psi}
\lim\limits_{n\rightarrow \infty}
\frac{\frac{1}{n}\sum\limits_{k=n+1}^{\infty}k\psi(k)}{\psi(n) } =0, \ \ \ \psi\in\mathcal{D}_{0},
\end{equation}
яка була встановлена в  \cite[с. 553]{SerdyukStepanyuk_UMJ_1_2023}.

%Формули \eqref{interpolation_D0_Inequality} i \eqref{interpolation_D0_Equality} випливають безпосередньо з \eqref{interpolation_generalCase_C^psi_Inequality}, \eqref{interpolation_generalCase_C^psi_Equality}, \eqref{Lemma_SumInterpolationPart_Inequality} та з доведення наслідку 3.1 роботи \cite{SerdyukStepanyuk_UMJ_1_2023}.

\end{proof}

Зауважимо, що вперше асимптотично непокращувані нерівності типу Лебега для інтерполяційних поліномів Лагранжа на множинах $C^{\psi}_{\beta}L_{1}$ за умови $\psi\in\mathcal{D}_{0}$ були одержані в \cite{Serdyuk2012}. Там же були наведені зазначені нерівності і на множинах узагальнених інтегралів Пуассона $C^{\psi}_{\beta}L_{1}$ при $\alpha>0$, $r>1$, $\beta\in \mathbb{R}$. При цьому вказані результати одержуються як частинний випадок ($p=1$) більш загальних інтерполяційних нерівностей типу Лебега, які формулюються для множин $C^{\psi}_{\beta}L_{p}$, $1\leq p \leq\infty$, $\psi\in\mathcal{D}_{0}$.

Зазначимо також, що можна навести приклади функцій $\psi(k)$, для яких
\begin{equation}\label{CauschiConditiion=0}
\lim\limits_{k\rightarrow \infty} \sqrt[k]{\psi(k)} = 0,
\end{equation}
які не задовольняють умову $\mathcal{D}_{0}$ і для яких виконується граничне співвідношення \eqref{LimRelation}. В силу \eqref{CauschiConditiion=0}, як випливає з \cite[с. 141]{Stepanets1} множини $C^{\psi}_{\beta}L_{1}$  є множинами цілих функцій і для них в силу \eqref{LimRelation} теорема~\ref{Interpolation_2023_theorem_GeneralCase} дозволяє записати асимптотично точні нерівності Лебега для інтерполяційних поліномів Лагранжа, не зважаючи на те, що 
$\psi \notin \mathcal{D}_{0}$.

%\section{Розв'язок задачі Колмогорова-Нікольського для інтерполяційних поліномів Лагранжа на класах нескінченно диференційовних функцій $C^{\psi}_{\beta,1}$ }

\section{Розв'язок задачі Колмогорова--Нікольського для інтерполяційних поліномів Лагранжа на класах згорток нескінченно диференційовних, аналітичних   та цілих функцій }\label{corrolarySection_analyticFunctions}

Наведемо приклади важливих функціональних компактів $C^{\psi}_{\beta,1}$, для яких формула \eqref{sup_interpolation_generalCase_C^psi_Equality} з теореми~\ref{sup_Interpolation_2023_theorem_GeneralCase} дозволяє записати асимптотичні  рівності для $\tilde{\mathcal{E}}_{n}(C^{\psi}_{\beta,1}; x)_{C}$ при $n\rightarrow\infty$.

\subsection{Наслідки з теореми~\ref{sup_Interpolation_2023_theorem_GeneralCase} для класів нескінченно диференційовних функцій}\label{Subsection_corrolarySection_analyticFunctions}

В наступному твердженні міститься асимптотична рівність для величин $\tilde{\mathcal{E}}_{n}(C^{\psi}_{\beta,1};x)_{C}$, $\psi\in \mathfrak{M}^{\alpha}$ при $n\rightarrow\infty$ за умови $\lambda(t)\uparrow \infty$.

\begin{theorem}\label{theorem_sup_Interpolation}
Нехай $\beta\in\mathbb{R}$, $\psi\in\mathfrak{M}^{\alpha}$ i для характеристик $\alpha(t)$ i $\lambda(t)$ означених формулами \eqref{psi_alpha} i  \eqref{psi_lambda}  відповідно
виконуються  умови \eqref{alphaTo_0} i \eqref{lambdaTo_infty}.
Тоді при всіх $n\in\mathbb{N}$  підпорядкованих умові \eqref{alpha_1/4}
має місце оцінка
\begin{equation}\label{Theorem_sup_psi}
\tilde{\mathcal{E}}_{n}(C^{\psi}_{\beta,1};x)=
\frac{2}{\pi} \left| \sin \frac{2n-1}{2}x\right| 
\psi(n) \lambda(n)
\left( 
 1+ \xi_{1}^{*}\alpha(n)+ \frac{\xi_{2}^{*}}{\lambda(n)}
\right),
\end{equation}
де  $-4(1+\pi)\leq \xi_{1}^{*}\leq \frac{4}{3}(2+\pi)$,  $-(1+\pi)\leq \xi_{2}^{*}\leq 2+\pi$.
\end{theorem}

\begin{proof}[Доведення теореми~\ref{theorem_sup_Interpolation}]

За виконання умов теореми~\ref{theorem_sup_Interpolation} $\psi\in\mathfrak{M}^{\alpha}$ тому, як випливає із \eqref{Mathfrak_M_alpha}, справджуються всі умови теореми~\ref{sup_Interpolation_2023_theorem_GeneralCase}, згідно з якою має місце рівність \eqref{sup_interpolation_generalCase_C^psi_Equality}.

Далі, як зазначалось при доведенні теореми~\ref{sup_Interpolation_2023_theorem_GeneralCase} за виконання умов \eqref{alphaTo_0} i \eqref{lambdaTo_infty} при всіх номерах $n$, що задовольняють нерівність \eqref{alpha_1/4} мають місце оцінки \eqref{Sum_psi(k)} i \eqref{Sum_kpsi(k+n)}. Тому в силу 
\eqref{sup_interpolation_generalCase_C^psi_Equality} для довільних $x\in\mathbb{R}$
\begin{align}\label{Theorem_sup_psi_proof}
&\tilde{\mathcal{E}}_{n}(C^{\psi}_{\beta,1};x) \!= \!
\frac{2}{\pi} \left| \sin \frac{2n\!-\!1}{2}x\right| \!\!
\left(
\psi(n) \lambda(n) \!\!
\left( \!
 1\!+\! \Theta_{4}\alpha(n)+\! \!\frac{\Theta_{5}}{\lambda(n)}
\right) \!+\!
\Theta\psi(n)\lambda(n)\left( 
 \Theta_{6}\alpha(n)\!+\! \frac{\Theta_{7}}{\lambda(n)}
\right) \!\! \right) \notag \\
=&
\frac{2}{\pi} \left| \sin \frac{2n-1}{2}x\right| 
\psi(n) \lambda(n)
\left( 
 1\!+\! (\Theta_{4}+\Theta\Theta_{6})\alpha(n)+ (\Theta_{5}+\Theta\Theta_{7}) \frac{1}{\lambda(n)}
\right).
\end{align}

Оцінка \eqref{Theorem_sup_psi} випливає із \eqref{Theorem_sup_psi_proof}. Теорему~\ref{theorem_sup_Interpolation} доведено.

\end{proof}

  Наведемо наслідок з теореми~\ref{theorem_sup_Interpolation} у випадку, коли $\psi(t)=e^{-\alpha t^{-r}}$,  $\alpha>0$, $0<r< 1$, тобто коли класи $C^{\psi}_{\beta,1} $ є класами узагальнених інтегралів Пуассона 
$C^{\alpha,r}_{\beta,1}$. Як уже зазначалось раніше, в цьому випадку характеристики $\alpha(t)$ i $\lambda(t)$, $t\geq 1$, мають вигляд \eqref{lambda_alpha_generalizedPoisson_Integtals}, а нерівність \eqref{alpha_1/4} еквівалентна нерівності  $n\geq \left(\frac{4}{\alpha r} \right)^{\frac{1}{r}}$. Тому із теореми~\ref{theorem_sup_Interpolation}  одержуємо наступне твердження

\begin{corollary}\label{cor00}
Нехай $0<r<1$, $\alpha>0$, $\beta\in\mathbb{R}$, $n \in \mathbb{N}$. Тоді при всіх  $x\in\mathbb{R}$ і номерів $n$ таких, що $n\geq \left(\frac{4}{\alpha r} \right)^{\frac{1}{r}}$ справедлива рівномірна по  всіх розглядуваних параметрах оцінка
\begin{equation}\label{f112}
\tilde{\mathcal{E}}_{n}(C^{\alpha,r}_{\beta,1};x)= 
\frac{2}{\pi} \left| \sin \frac{2n-1}{2}x\right|
e^{-\alpha n^{r}}   n^{1-r} \Big(
\frac{1}{\pi\alpha r }+\mathcal{O}(1)\Big(\frac{1}{(\alpha r)^{2}}\frac{1}{n^{r}}+\frac{1}{n^{1-r}}\Big)\Big).
\end{equation}
\end{corollary}

Зазначимо, що асимптотичні рівності  для величин $\tilde{\mathcal{E}}_{n}(C^{\alpha,r}_{\beta,p};x)$, $r\in(0,1)$ при \cite{StepanetsSerdyuk2000Zb},
\cite{L_Serdyuk_2004_nd},
\cite{SerdyukStepanyuk2023_UMJ_No7}.

Наведемо ще декілька прикладів застосування теореми~\ref{theorem_sup_Interpolation} для різних функцій $\psi$ із $\mathfrak{M}$, які задовольняють умовам \eqref{alphaTo_0} і \eqref{lambdaTo_infty}. А саме розглянемо $\psi$ вигляду \eqref{psi_1}, \eqref{psi_2}, \eqref{psi_3}.

Для зазначених функцій $\psi(t)$  характеристики 
$\lambda(t)$  i $\alpha(t)$  відображено в наступній таблиці:

\begin{center}\label{Tabl}
%\begin{tiny}
%\begin{table}[H]
%\begin{center}
%\begin{large}
\begin{tabular}{|c|c|c|c|c|}
	\hline
№ &  $\psi(t)$ &  $ \alpha(t)$ & $\lambda(t)$    \\
\hline
1. & $(t+2)^{- \ln\ln (t+2)}$ & $\frac{t+2}{t}\frac{1}{1+\ln\ln(t+2)}$  & $\frac{t+2}{1+\ln\ln(t+2)}$   \\
&  &   &    \\
\hline
2. & $e^{- \ln^{2} (t+1)}$ & $\frac{t+1}{t}\frac{1}{2\ln(t+1)}$  & $\frac{t+1}{2\ln(t+1)}$   \\
&  &   &    \\
\hline
3. & $e^{- \frac{t+2}{\ln (t+2) }}$ &  $\frac{\ln^{2}(t+2)}{t(\ln(t+2)-1)}$ & $\frac{\ln^{2}(t+2)}{\ln(t+2)-1}$  \\
&  &   &    \\
\hline
\end{tabular}
%\end{large}
\end{center}
%\end{table}

Із теореми~\ref{theorem_sup_Interpolation} з урахуванням наведених в таблиці значень $\alpha(t)$ i $\lambda(t)$, отримуємо твердження.

\begin{corollary}\label{cor3}
Нехай $\psi(k)=(k+2)^{- \ln\ln (k+2)}$, $k=1,2,...,$ $\beta\in\mathbb{R}$ і $n\in\mathbb{N}$. Тоді для всіх $x\in\mathbb{R}$ при $n\rightarrow\infty$ виконується  асимптотична рівність
\begin{equation}\label{f13}
\tilde{\mathcal{E}}_{n}(C^{\psi}_{\beta,1};x) =
\frac{2}{\pi} \left| \sin \frac{2n-1}{2}x\right| 
\psi(n)\frac{n}{\ln\ln (n+2)}\left( 1+\mathcal{O}(1)\frac{n}{\ln \ln(n+2)}\right).
\end{equation}
\end{corollary}

\begin{corollary}\label{cor4}
Нехай  $\psi(k)=e^{-\ln ^{2}(k+1) }$,  $k=1,2,...,$ $\beta\in\mathbb{R}$ і $n\in\mathbb{N}$. Тоді  для всіх $x\in\mathbb{R}$ при $n\rightarrow\infty$ має місце  асимптотична рівність
\begin{equation}\label{f16}
\tilde{\mathcal{E}}_{n}(C^{\psi}_{\beta,1};x) =
\frac{1}{\pi}\left| \sin \frac{2n-1}{2}x\right| \frac{\psi(n) n}{\ln (n+1)} \left( 1
+\mathcal{O}(1)\frac{1}{\ln(n+1)}\right) .
\end{equation}
\end{corollary}

\begin{corollary}\label{cor5}
Нехай $\psi(k)=e^{- \frac{k+2}{\ln (k+2)}}$, $k=1,2,...,$ $\beta\in\mathbb{R}$ і $n\in\mathbb{N}$. Тоді для всіх $x\in\mathbb{R}$  при $n\rightarrow\infty$ має місце  асимптотична рівність
\begin{equation}\label{f20}
\tilde{\mathcal{E}}_{n}(C^{\psi}_{\beta,1};x) =
\frac{2}{\pi} \left| \sin \frac{2n-1}{2}x\right| \psi(n)\ln (n+2) \left( 1+\mathcal{O}(1)\frac{1}{\ln (n+2)} \right).
\end{equation}
\end{corollary}

У формулах \eqref{f13}-\eqref{f20} величини $\mathcal{O}(1)$ рівномірно обмежені по усіх розглядуваних параметрах.
  
Наведемо наслідки з теореми~\ref{sup_Interpolation_2023_theorem_GeneralCase} у випадку, коли $\psi(k)=e^{-\alpha k}$, $\alpha>0$, тобто коли класи $C^{\psi}_{\beta,1}$ є класами інтегралів Пуассона $C^{\alpha,1}_{\beta,1}$.

\begin{corollary}\label{cor6}
Нехай  $\alpha>0$ i $\beta\in\mathbb{R}$. Тоді при всіх  $x\in\mathbb{R}$ і довільних $n \in \mathbb{N}$ має місце рівність
\begin{equation}\label{cor6_equality}
\tilde{\mathcal{E}}_{n}(C^{\alpha,1}_{\beta,1};x)= 
\frac{2}{\pi} \left| \sin \frac{2n-1}{2}x\right|
e^{-\alpha n}   \Big(
\frac{1}{1-e^{-\alpha} }+\frac{\Theta}{n}   \frac{e^{-\alpha}}{(1-e^{-\alpha} )^{2}} \Big),
\end{equation}
де для величини $\Theta=\Theta(n;\alpha;\beta;x)$ виконуються нерівності $-(1+\pi)\leq\Theta\leq 1$.
\end{corollary}
\begin{proof}
Покладемо у формулі
\eqref{sup_interpolation_generalCase_C^psi_Equality} теореми~\ref{sup_Interpolation_2023_theorem_GeneralCase} $\psi(k)=e^{-\alpha k}$, $\alpha>0$ і використаємо рівність \eqref{Sums_generalizedPoissonIntegrals_2}. Будемо мати оцінки
\begin{align}\label{cor6_proof}
\tilde{\mathcal{E}}_{n}(C^{\alpha,1}_{\beta,1};x)=& 
\frac{2}{\pi} \left| \sin \frac{2n-1}{2}x\right|
 \Big(\sum\limits_{k=n}^{\infty} e^{-\alpha k}+ \frac{\Theta}{n}\sum\limits_{k=1}^{\infty} ke^{-\alpha (k+n)} \Big)\notag \\
=&
\frac{2}{\pi} \left| \sin \frac{2n-1}{2}x\right|
   \Big(
\frac{e^{-\alpha n}}{1-e^{-\alpha} }+\frac{\Theta}{n}   \frac{e^{-\alpha(n+1)}}{(1-e^{-\alpha} )^{2}} \Big).
\end{align}
З \eqref{cor6_proof} випливає \eqref{cor6_equality}. Наслідок~\ref{cor6} доведено.
\end{proof}

Нехай, далі $\psi(k)$, що породжують класи $C^{\psi}_{\beta,1}$, задовольняють уомву $\mathcal{D}_{q}$ вигляду \eqref{DalamberCondition} при $q\in(0,1)$. Тоді із теореми~\ref{sup_Interpolation_2023_theorem_GeneralCase} одержуємо такий наслідок.

\begin{corollary}\label{cor7}
Нехай $\psi\in \mathcal{D}_{q}$, $q\in(0,1)$,  $\beta\in\mathbb{R}$. Тоді для всіх  $x\in\mathbb{R}$  при всіх $n \in \mathbb{N}$ таких, що задовольняють нерівність \eqref{Number_n_Dq}, має місце оцінка
\begin{equation}\label{cor7_equality}
\tilde{\mathcal{E}}_{n}(C^{\psi}_{\beta,1};x)= 
 \left| \sin \frac{2n-1}{2}x\right|
\psi(n)   \Big(
\frac{2}{\pi(1-q)} +  \mathcal{O}(1) \left( \frac{q}{n(1-q)^{2}}+ \frac{\varepsilon_{n}}{n(1-q)^{2}}   \right) \Big),
\end{equation}
в якій $\varepsilon_{n}$ означена формулою \eqref{Varepselon_n_Dq}, a $\mathcal{O}(1)$ --- рівномірно обмежена по всіх параметрах величина.
\end{corollary}

Оцінка \eqref{cor7_equality} є наслідком з теореми~\ref{sup_Interpolation_2023_theorem_GeneralCase} і рівності \eqref{Formulas_Corrolary}.

Як зазначалось в п.~\ref{Subsection_analytic_Section_corrolary_infinitelyDifferentiable}, прикладами ядер $\Psi_{\beta}(t)$ вигляду \eqref{kernelPsi}, коефіцієнти   $\psi(k)$ яких задовольняють умову $\mathcal{D}_{q}$, $q\in(0,1)$ є полігармонічні ядра Пуассона $P_{q,\beta}(\ell,t)$ вигляду \eqref{3Kern_PUAS_POLI}, ядра аналітичних функцій  $\mathcal{P}_{q,\beta}(t)$ вигляду \eqref{3Kern_PUAS_TEPL},  ядра Неймана $N_{q,\beta}(t)$ вигляду \eqref{3Kern_Neim} та ін. На підставі формули \eqref{cor7_equality} для класів згорток $C^{\psi}_{\beta,1}$, породжених уісма переліченими ядрами, можна записати асимптотичні рівності величин $\tilde{\mathcal{E}}_{n}(C^{\psi}_{\beta,1};x)$, $x\in\mathbb{R}$ при $n\rightarrow\infty$. Вперше асимптотичні рівності для величин виду \eqref{quantityInterpol} на таких класах були одержані в роботі \cite{Serdyuk2012}. 

Приклад функції $\psi(k)$ вигляду \eqref{psi_even_oddCase} показує, що теорема~\ref{sup_Interpolation_2023_theorem_GeneralCase} дозволяє знаходити сильну асимптотику величин $\tilde{\mathcal{E}}_{n}(C^{\psi}_{\beta,1};x)$, $x\in\mathbb{R}$  при $n\rightarrow\infty$ і для випадків, коли ні умова $\mathcal{D}_{q}$ вигляду \eqref{DalamberCondition}, ні умова Коші $\mathcal{C}_{q}$ вигляду \eqref{CauschiConditiion} не виконуються. В зазначеному випадку в силу \eqref{sup_interpolation_generalCase_C^psi_Equality}, \eqref{Sum_psi(k)_psi_even_oddCase} i \eqref{1/nSum_kpsi(k+n)_psi_even_oddCase} має місце асимптотична при $n\rightarrow \infty$ рівність
\begin{equation}\label{cor_exceptionCase}
\tilde{\mathcal{E}}_{n}(C^{\psi}_{\beta,1};x)= 
\frac{2}{\pi} \left| \sin \frac{2n-1}{2}x\right|
\frac{\mathrm{max} \left\{\psi(n), \psi(n+1) \right\} }{1-\left( \overline{\lim\limits_{k\rightarrow\infty}} (\psi(k))^{\frac{1}{k}}\right)^{2}}
\left(1+  \frac{\mathcal{O}(1)}{n  \left(1- \left( \overline{\lim\limits_{k\rightarrow\infty}} (\psi(k))^{\frac{1}{k}}\right)^{2} \right)} \right).
\end{equation}
Формула \eqref{cor_exceptionCase} є наслідком з рівності \eqref{sup_interpolation_generalCase_C^psi_Equality} та співвідношень \eqref{Sum_psi(k)_psi_even_oddCase}, \eqref{1/nSum_kpsi(k+n)_psi_even_oddCase} i  \eqref{Sum_psi(k)_psi_even_oddCase_LimitRelations}. Нехай далі $\psi(k)$ задовольняє умову $\mathcal{D}_{0}$ вигляду \eqref{D0Condition}. Для таких $\psi$ виконуються формули \eqref{interpolation_D0_Equality_ProofEq1} i \eqref{LimitRelations_Sum/psi} і тому із теореми~\ref{sup_Interpolation_2023_theorem_GeneralCase} одержуємо 
\begin{corollary}\label{cor8_D0}
Нехай $\psi\in \mathcal{D}_{0}$,   $\beta\in\mathbb{R}$ i $n \in \mathbb{N}$. Тоді при всіх  $x\in\mathbb{R}$   має місце асимптотична при $n\rightarrow\infty$ рівність
\begin{equation}\label{cor8D0_equality}
\tilde{\mathcal{E}}_{n}(C^{\psi}_{\beta,1};x)= 
\frac{2}{\pi} \left| \sin \frac{2n-1}{2}x\right|
  \Big(
\psi(n) + \frac{ \mathcal{O}(1)}{n} \sum\limits_{k=n+1}^{\infty}k \psi(k)\Big),
\end{equation}
в якій  $\mathcal{O}(1)$ рівномірно обмежена по всіх параметрах величина.
\end{corollary}

Зазначимо, що вперше асимптотичні рівності для величини $\tilde{\mathcal{E}}_{n}(C^{\psi}_{\beta,1};x)$ при $n\rightarrow\infty$ із більш тонкою оцінкою залишкового члена були одержані в \cite{Serdyuk2012}. Там же були записані асимптотичні рівності величин \eqref{quantityInterpol} і на класах узагальнених інтегралів Пуассона $C^{\alpha,r}_{\beta,1}$ при $\alpha>0$ i $r>1$.

Варто зауважити, що формула \eqref{sup_interpolation_generalCase_C^psi_Equality} іноді дозволяє записувати асимптотичні рівності для величин $\tilde{\mathcal{E}}_{n}(C^{\psi}_{\beta,1};x)$ для функцій $\psi$, які задовольняють умову \eqref{CauschiConditiion=0} i водночас не задовольняють умову $\mathcal{D}_{0}$.

Співставлення отриманих у даному розділі результатів для величин $\tilde{\mathcal{E}}_{n}(C^{\psi}_{\beta,1};x)$  --- точних верхніх меж відхилень інтерполяційних поліномів на класах $C^{\psi}_{\beta,1}$ із отриманими в \cite{SerdyukStepanyuk_UMJ_1_2023} результатами щодо аналогічних величин ${\mathcal{E}}_{n}(C^{\psi}_{\beta,1})$ 
для частинних сум Фур'є $S_{n-1}$ 
\begin{equation}\label{FourierApproximation}
\mathcal{E}_{n}(C^{\psi}_{\beta,1}) =
\sup\limits_{f\in C^{\psi}_{\beta,1}} |f(x) - S_{n-1}(f;x)|
\end{equation}
показує, що за умови \eqref{LimRelation} величини \eqref{quantityInterpol} i \eqref{FourierApproximation} пов'язані граничним співвідношенням
\begin{equation*}
\lim\limits_{n\rightarrow\infty}\frac{  \tilde{\mathcal{E}}_{n}(C^{\psi}_{\beta,1};x) }{ 2 \left| \sin \frac{2n-1}{2}x\right| \mathcal{E}_{n}(C^{\psi}_{\beta,1})}=1.
\end{equation*}

\vspace{4mm}

Дана робота частково підтримана грантом H2020-MSCA-RISE-2019, номер проєкту 873071 (SOMPATY: Spectral Optimization: From Mathematics to Physics and Advanced Technology), а також фондом Фольсквагена (VolkswagenStiftung),  програмою “From Modeling and Analysis to Approximation”.

\end{document}